\documentclass[12pt]{amsart}

\usepackage{amsmath}
\usepackage{amssymb}
\usepackage{ascmac}
\usepackage{amsthm}

\makeatletter

\@addtoreset{equation}{section}
\makeatother 

\theoremstyle{plain}
\newtheorem{thm}{Theorem}[section]
\newtheorem*{thm*}{Theorem}
\newtheorem{prop}[thm]{Proposition}
\newtheorem{lem}[thm]{Lemma}
\newtheorem{cor}[thm]{Corollary}

\theoremstyle{definition}

\theoremstyle{remark}
\newtheorem{rem}{Remark}[section]

\newcommand{\ric}{\operatorname{Ric}}
\newcommand{\Div}{\operatorname{div}}
\newcommand{\Hess}{\operatorname{Hess}}

\newcommand{\bm}{\partial M}
\newcommand{\inte}{\mathrm{Int}\,}

\newcommand{\IR}{\operatorname{InRad}}

\newcommand{\med}{\operatorname{med}}

\newcommand{\CD}{\operatorname{CD}}

\usepackage{latexsym}

\title[Upper bounds for Poincar\'e constants]{Upper bounds for higher-order Poincar\'e constants}

\author{Kei Funano}
\address{Division of Mathematics \& Research Center for Pure and Applied Mathematics, Graduate School of Information Sciences, Tohoku University, 6-3-09 Aramaki-Aza-Aoba, Aoba-ku, Sendai 980-8579, Japan}
\email{kfunano@tohoku.ac.jp}

\author{Yohei Sakurai}
\address{Advanced Institute for Materials Research,Tohoku University, 2-1-1 Katahira, Aoba-ku, Sendai, 980-8577, Japan}
\email{yohei.sakurai.e2@tohoku.ac.jp}

\subjclass[2010]{35P15, 53C23, 58J50}
\keywords{Poincar\'e constant; Eigenvalue of the Laplacian; Multi-way isoperimetric constant}
\date{October 27, 2019}

\begin{document}
\maketitle

\begin{abstract}
We introduce higher-order Poincar\'e constants for compact weighted
 manifolds and estimate them from
 above in terms of subsets. These estimates imply upper bounds
 for eigenvalues of the weighted Laplacian and the first nontrivial eigenvalue of
 the $p$-Laplacian. In the case of the closed eigenvalue problem and the Neumann eigenvalue problem these are related with the estimates obtained by Chung-Grigor'yan-Yau and Gozlan-Herry. We also obtain similar upper
 bounds for Dirichlet eigenvalues and multi-way isoperimetric
 constants. As an application,
 for manifolds with boundary of non-negative dimensional weighted Ricci curvature,
 we give upper bounds for inscribed radii in terms of dimension and the first Dirichlet Poincar\'e constant.
\end{abstract}

%%%%%%%%%%%%%%%%%%%%%%%%%%%%
%%%%%%%%%%%%%%%%%%%%%%%%%%%%
%%%%%%%%%%%%%%%%%%%%%%%%%%%%
\section{Introduction}Let $M=(M,m)$ be a compact weighted Riemannian manifold,
namely,
$M=(M,g)$ is a (connected) compact Riemannian manifold equipped with the
Riemannian distance $d$ determined by $g$ and $m$ is a (probability) weighted Riemannian volume measure defined as
\begin{equation}\label{eq:weighted volume measure}
m:=e^{-f} \,v_{g}
\end{equation}
for a smooth function $f \in C^{\infty}(M)$ such that $\int_{M}\,e^{-f} \,dv_{g}=1$,
where $v_{g}$ denotes the Riemannian volume measure induced from $g$.
The purpose of this article is to present a unifying way to give upper bounds
for eigenvalues of the Laplacian and the $p$-Laplacian with or without
a boundary condition.
%%%%%%%%%%%%%%%%%%%%%%%%%%
\subsection{Closed manifolds}
We first consider the case where
$M$ is a closed manifold (i.e., its boundary $\partial M$ is empty).
We denote by $\nabla$ and by $\Vert \cdot \Vert$ the gradient operator
and the canonical norm induced from $g$ respectively.
For $k\geq 1$ and $p\in [1,\infty)$,
one of our main objects is the \textit{$k$-th (Neumann) $p$-Poincar\'e constant} defined by
\begin{equation}\label{eq:Poincare constant}
\nu_{k,p}(M,m):=\inf_{L_{k}} \sup_{\phi \in L_{k}\setminus \{0\}} \frac{\int_{M} \,\Vert \nabla \phi \Vert^{p} \,   dm  }{ \int_{M} \,\vert \phi-\int_{M}\,\phi\,dm \vert^{p} \,   dm },
\end{equation}
where the infimum is taken over all $k$-dimensional subspaces $L_{k}$ of the $(1,p)$-Sobolev space $W^{1,p}(M,m)$ with $L_{k}\cap\, \mathcal{C}=\{0\}$ for the set $\mathcal{C}$ of all constant functions on $M$.
In the case of $p=2$,
this Poincar\'e constant is equal to the $k$-th non-trivial eigenvalue of the
weighted Laplacian on $M$ due to the min-max principle. In the case of $p\in (1,\infty)$, the value $\nu_{1,p}(M,m)$ is equivalent to the first eigenvalue of the weighted $p$-Laplacian (more precisely, see Subsection \ref{sec:Properties of Poincare constants}).
Furthermore,
in the case of $p=1$,
the value $\nu_{1,1}(M,m)$ is known to be equivalent to the so-called Cheeger isoperimetric constant (see Subsection \ref{sec:isoperimetric Closed manifolds}).

We now assert one of our main results. For any finite sequence
$\{A_{\alpha}\}_{\alpha=0}^k$ of Borel subsets of $M$ we set
\begin{align*}
\mathcal{D}(\{A_{\alpha}\}):=\min_{\alpha \neq \beta}
d(A_{\alpha},A_{\beta}),
 \end{align*}where $d(A_{\alpha},A_{\beta}):=\inf \{d(x,y) \mid x\in
 A_{\alpha},y\in A_{\beta}\}$.
\begin{thm}\label{thm:main result}
Let $(M,m)$ be a closed weighted Riemannian manifold. For any sequence 
 $\{A_{\alpha}\}^{k}_{\alpha=0}$ of Borel subsets of $M$ we have
\begin{equation}\label{eq:main inequality}
 \nu_{k,p}(M,m)^{\frac{1}{p}}  \leq \frac{2}{ \mathcal{D}(\{A_{\alpha}\})  }\max_{\alpha=0,\dots,k}  \log \frac{e(1-\sum_{\beta\neq \alpha}m(A_{\beta})    )}{m(A_{\alpha})}.
\end{equation}
\end{thm}

\begin{rem}\label{rem:previous result in linear case}
In the case where $p=2$,
this type of inequality has been obtained by Chung-Grigor'yan-Yau.
Under the same setting as in Theorem \ref{thm:main result},
they \cite{CGY2} have proven the following inequality (see Theorem 3.1 in \cite{CGY2}, and see also \cite{CGY1}):
\begin{equation}\label{eq:CGY}
 \nu_{k,2}(M,m)^{\frac{1}{2}} \leq \frac{1}{  \mathcal{D}(\{A_{\alpha}\}) }\max_{\alpha\neq \beta}  \log \frac{e}{m(A_{\alpha})m(A_{\beta})}.
\end{equation}
When $k=1$,
this inequality (\ref{eq:CGY}) is due to Gromov-Milman \cite{GM}. The proof of (\ref{eq:CGY}) in \cite{CGY1} is based on the eigenfunction expansion of the heat kernel,
and a heat kernel estimate.
Recently,
Gozlan-Herry \cite{GH} shown a similar inequality to (\ref{eq:CGY}) in a
 different way from that in \cite{CGY1} based on a simple geometric
 observation (see Proposition 2.2 in \cite{GH}). In Remark \ref{rem:comparing} we
 will compare our inequality (\ref{eq:main inequality}) with the inequalities of
 Chung-Grigor'yan-Yau (\ref{eq:CGY}) and Gozlan-Herry.
\end{rem}

Let us mention our method of the proof of Theorem \ref{thm:main result}.
We prove the desired inequality by combining the following three principles,
and a geometric observation:
(1) domain monotonicity principle for a domain $\Omega$ in $M$,
which claims that
the $(k+1)$-th Dirichlet $p$-Poincar\'e constant on $\Omega$ is at least a modification of $\nu_{k,p}(M,m)$ (see Lemma \ref{lem:domain monotonicity});
(2) domain decomposing principle for a pairwise disjoint sequence $\{\Omega_{\alpha}\}^{k}_{\alpha=0}$ of domains,
which is a relation between the first Dirichlet $p$-Poincar\'e constants on the $\Omega_{\alpha}$,
for $\alpha=0,1,\dots,k$,
and the $(k+1)$-th one on $\sqcup^{k}_{\alpha=0} \Omega_{\alpha}$ (see Lemma \ref{lem:higher order to smallest});
(3) boundary concentration inequality for a domain $\Omega$,
which states that
the restricted weighted measure on $\Omega$ exponentially concentrates around its boundary when the first Dirichlet $p$-Poincar\'e constant on $\Omega$ is large (see Lemma \ref{lem:local boundary concentration inequality}).
We notice that
our method is quite different from that of (\ref{eq:CGY}) in \cite{CGY1}.

%%%%%%%%%%%%%%%%%%%%%%%%%%
\subsection{Compact manifolds with boundary}
We next consider the case where
$M$ is a compact manifold with boundary.
In this case,
our main object is the \textit{$k$-th Dirichlet $p$-Poincar\'e constant} defined by
\begin{equation}\label{eq:Dirichlet Poincare constant}
\nu^{D}_{k,p}(M,m):=\inf_{L_{k}} \sup_{\phi \in L_{k}\setminus \{0\}} \frac{\int_{M} \,\Vert \nabla \phi \Vert^{p} \,   dm  }{ \int_{M} \,\vert \phi \vert^{p} \,   dm },
\end{equation}
where the infimum is taken over all $k$-dimensional subspaces $L_{k}$ of the $(1,p)$-Sobolev space $W^{1,p}_{0}(M,m)$.
When $p=2$,
this constant is equal to the $k$-th Dirichlet eigenvalue of the weighted Laplacian.
For $p\in (1,\infty)$
the value $\nu^{D}_{1,p}(M,m)$ coincides with the first Dirichlet eigenvalue of the weighted $p$-Laplacian (see Subsection \ref{sec:Dirichlet Poincare constants}),
and $\nu^{D}_{1,1}(M,m)$ the Dirichlet isoperimetric constant (see Subsection \ref{sec:isoperimetric manifolds with boundary}).

Our method of the proof of Theorem \ref{thm:main result} also works in this case where $\partial M$ is non-empty,
and this yields the following analogue of Theorem \ref{thm:main result}
for the Dirichlet Poincar\'e constant. For any sequence
$\{A_{\alpha}\}_{\alpha=1}^k$ of Borel subsets of $M$ we set
\begin{equation}\label{eq:pre boundary separation distance}
 \mathcal{D}^{\partial}\bigl(\{A_{\alpha}\}\bigl):=\min \{\, \min_{\alpha \neq \beta} d(A_{\alpha},A_{\beta}),\,\, \min_{\alpha} d(A_{\alpha},\partial M)  \,\}.
 \end{equation}

\begin{thm}\label{thm:main result with boundary}
Let $(M,m)$ be a compact weighted Riemannian manifold with boundary.
For any sequence $\{A_{\alpha}\}^{k}_{\alpha=1}$ of Borel subsets of $M$,
\begin{equation}\label{eq:main inequality with boundary}
\nu^{D}_{k,p}(M,m)^{\frac{1}{p}}\leq \frac{2}{   \mathcal{D}^{\partial}(\{A_{\alpha}\})   }\max_{\alpha=1,\dots,k}  \log \frac{e(1-\sum_{\beta\neq \alpha}m(A_{\beta})    )}{m(A_{\alpha})}.
\end{equation}
\end{thm}
The part $\mathcal{D}^{\partial}(\{A_{\alpha}\})$ in the right hand side of (\ref{eq:main inequality with boundary}) concerns boundary concentration phenomena studied in \cite{S}.
Similarly to Theorem \ref{thm:main result},
this estimate (\ref{eq:main inequality with boundary}) is new for $p\neq 2$.
When $k\geq 2$,
Theorem \ref{thm:main result with boundary} is new even in the case of $p=2$ in view of the following two remarks. 

\begin{rem}\label{rem:historical remark}
 Under the same setting as in Theorem \ref{thm:main result with boundary},
the second author \cite{S} has obtained the following weaker estimate (see Lemma 4.1 in \cite{S}):
\begin{equation}\label{eq:weaker main result with boundary}
\nu^{D}_{k,2}(M,m)^{\frac{1}{2}} \leq \frac{2}{ \mathcal{D}^{\partial}(\{A_{\alpha}\}) }\frac{1}{\sqrt{\min_{\alpha=1,\dots,k} m(A_{\alpha})}}.
\end{equation}
After that
the authors \cite{FS} improved (\ref{eq:weaker main result with boundary}),
and proved Theorem \ref{thm:main result with boundary} only when $k=1$ (see Theorem 2.3 in \cite{FS}).
\end{rem}
\begin{rem}\label{rem:compare with CGY}
We now compare Theorem \ref{thm:main result with boundary} for $p=2$ with Theorem 1.1 in \cite{CGY1}.
Chung-Grigor'yan-Yau \cite{CGY1} have obtained an upper estimate of the $k$-th Robin eigenvalues with same proof as (\ref{eq:CGY}).
In our setting,
their estimate holds in the following form:
Under the same setting as in Theorem \ref{thm:main result with boundary},
we have
\begin{align}\label{eq:CGY with boundary}
&\quad \,\,\left(\nu^{D}_{k,2}(M,m)-\nu^{D}_{1,2}(M,m)\right)^{\frac{1}{2}}\\ \notag
&\leq \frac{1}{\mathcal{D}(\{A_{\alpha}\})}\,\max_{\alpha \neq \beta} \,\log\,\frac{4}{\int_{A_{\alpha}} \,\phi^{2}_{1,2;m}\,dm \, \int_{A_{\beta}} \,\phi^{2}_{1,2;m}\,dm   },
\end{align}
where $\phi_{1,2;m}$ denotes an $L^{2}$-normalized eigenfunction of the first Dirichlet eigenvalue of the weighted Laplacian.
The advantage of Theorem \ref{thm:main result with boundary} is that
$\nu^{D}_{1,2}(M,m)$ and the integral quantity of the eigenfunction do not appear in  (\ref{eq:main inequality with boundary}).
We eliminate such values by considering $\mathcal{D}^{\partial}(\{A_{\alpha}\})$ instead of $\mathcal{D}(\{A_{\alpha}\})$.
The form of (\ref{eq:main inequality with boundary}) seems to be different from that of (\ref{eq:CGY with boundary}),
and more close to that of (\ref{eq:CGY}).
\end{rem}

\subsection{Organization}
In Section \ref{sec:Preliminaries},
we introduce the modified Poincar\'e constants and compare them with
the Poincar\'e constants (\ref{eq:Poincare constant}).

In Sections \ref{sec:Poincare constants and concentration}, \ref{sec:Dirichlet Poincare constants and boundary concentration}, \ref{sec:Multiway isoperimetric constants and concentration},
we prove our main results.
In Section \ref{sec:Poincare constants and concentration},
we prove Theorem \ref{thm:main result}.
In Section \ref{sec:Dirichlet Poincare constants and boundary concentration},
we formulate an analogue of Theorem \ref{thm:main result} for Dirichlet eigenvalues of the weighted $p$-Laplacian on compact manifolds with boundary (see Theorem \ref{thm:main result with boundary}).
In Section \ref{sec:Multiway isoperimetric constants and concentration},
we provide upper bounds for multi-way isoperimetric constant and also the multi-way Dirichlet isoperimetric constant (see Theorems \ref{thm:isoperimetric main result} and \ref{thm:Dirichlet isoperimetric main result}).

In Section \ref{sec:Sharpness},
we will discuss the sharpness of our main results.

Sections \ref{sec:Inscribed radius estimates} and \ref{sec:Discrete cases} are devoted to present some byproducts of the study in Sections \ref{sec:Poincare constants and concentration}, \ref{sec:Dirichlet Poincare constants and boundary concentration}, \ref{sec:Multiway isoperimetric constants and concentration}.
In Section \ref{sec:Inscribed radius estimates},
for compact manifolds with boundary of non-negative weighted Ricci curvature,
we give an upper bound of its inscribed radius in terms of the
Dirichlet Poincar\'e constant (\ref{eq:Dirichlet Poincare constant}) by using the domain monotonicity principle (Lemma \ref{lem:Dirichlet domain monotonicity}) and the boundary concentration inequality (Lemma \ref{lem:Dirichlet local boundary concentration inequality}) obtained in Section \ref{sec:Dirichlet Poincare constants and boundary concentration} (see Proposition \ref{prop:inscribed estimate}).
In Section \ref{sec:Discrete cases},
we discuss discrete cases.
We show an analogue of Theorem \ref{thm:main result} for weighted graphs (see Theorem \ref{thm:discrete modified main result}).

%%%%%%%%%%%%%%%%%%%%%%%%%%%
%%%%%%%%%%%%%%%%%%%%%%%%%%%
%%%%%%%%%%%%%%%%%%%%%%%%%%%
\section{Preliminaries}\label{sec:Preliminaries}
Hereafter,
let $(M,m)$ denote a compact weighted Riemannian manifold defined as (\ref{eq:weighted volume measure}).
Let $k\geq 1$ and $p\in [1,\infty)$.

%%%%%%%%%%%%%%%%%%%%%%%
\subsection{Modified Poincar\'e constants}\label{sec:Properties of Poincare constants}
Let $M$ be closed.
We define the \textit{modified $k$-th $p$-Poincar\'e constant} as
\begin{equation}\label{eq:modified Poincare constant}
\widehat{\nu}_{k,p}(M,m):=\inf_{L_{k+1}} \sup_{\phi \in L_{k+1}\setminus \{0\}} \frac{\int_{M} \,\Vert \nabla \phi \Vert^{p} \,   dm  }{ \int_{M} \,\vert \phi \vert^{p} \,   dm },
\end{equation}
where the infimum is taken over all $(k+1)$-dimensional subspaces $L_{k+1}$ of $W^{1,p}(M,m)$.
In stead of working with $\nu_{k,p}(M,m)$ we will work with
$\widehat{\nu}_{k,p}(M,m)$. Let us study the relation between the Poincar\'e constant $\nu_{k,p}(M,m)$ defined as (\ref{eq:Poincare constant}),
and the modified one $\widehat{\nu}_{k,p}(M,m)$.
In the case of $p=2$ it is known and easy to show that
$\nu_{k,2}(M,m)=\widehat{\nu}_{k,2}(M,m)$,
and they are also equal to the $k$-th non-trivial eigenvalue of the weighted Laplacian (see (\ref{eq:closed min max principle}) below).
We first recall the following elementary inequality (cf. Lemma 2.1 in \cite{Mi1}).
For the completeness of this paper,
we give its proof.
\begin{lem}\label{lem:Milman lemma}
For $\phi \in W^{1,p}(M,m)$,
we have
\begin{equation*}
\inf_{c\in \mathbb{R}}\Vert \phi-c \Vert_{L^{p}} \leq \left\Vert \phi-\int_{M}\,\phi\, dm \right \Vert_{L^{p}}\leq 2 \inf_{c\in \mathbb{R}}\Vert \phi-c \Vert_{L^{p}},
\end{equation*}
where $\Vert \cdot \Vert_{L^{p}}$ is the $L^{p}$-norm on $M$ with
 respect to the measure $m$.
\end{lem}
\begin{proof}
Given any $c\in \mathbb{R}$ we get
\begin{align*}
\left\Vert \phi-\int_{M}\,\phi\, dm \right \Vert_{L^{p}}&\leq \left \Vert \phi-c \right \Vert_{L^{p}}+\left\Vert \int_{M}\,\phi\, dm- c\right\Vert_{L^{p}}\\
                                                                               &  =   \left \Vert \phi-c \right \Vert_{L^{p}} +\left\vert \int_{M}\,\phi\, dm- c \right\vert\\
                                                                               & \leq \left\Vert \phi-c \right \Vert_{L^{p}}+\left\Vert \phi-c \right \Vert_{L^{1}}\leq 2 \left\Vert \phi-c \right \Vert_{L^{p}}.
\end{align*}
This proves the lemma.
\end{proof}

Lemma \ref{lem:Milman lemma} leads us to the following:
\begin{prop}\label{prop:equivalence between Poincare constant and modified Poincare constant}
\begin{equation*}
\nu_{k,p}(M,m) \leq \widehat{\nu}_{k,p}(M,m)\leq 2^{p}\, \nu_{k,p}(M,m).
\end{equation*}
\end{prop}
\begin{proof}
We first prove $\nu_{k,p}(M,m) \leq \widehat{\nu}_{k,p}(M,m)$.
Fix a $(k+1)$-dimensional subspace $L_{k+1}$ of $W^{1,p}(M,m)$,
and also define a subspace
\begin{equation*}
\overline{L}_{k+1}:=\left\{\phi \in L_{k+1}  \, \middle|\, \int_{M}\,\phi\,dm=0  \right\}.
\end{equation*}
Note that
the dimension of $\overline{L}_{k+1}$ is at least $k$,
and $\overline{L}_{k+1}\cap \mathcal{C}=\{0\}$ for the set $\mathcal{C}$ of all constant functions on $M$.
We now take a $k$-dimensional subspace $\overline{L}_{k}\subset \overline{L}_{k+1}$ of $W^{1,p}(M,m)$.
Then it holds that
\begin{align*}
\nu_{k,p}(M,m) &\leq \sup_{\phi \in \overline{L}_{k}\setminus \{0\}} \frac{\int_{M} \,\Vert \nabla \phi \Vert^{p} \,   dm  }{ \int_{M} \,\vert \phi-\int_{M}\,\phi\,dm \vert^{p} \,   dm }\\
                                          & =    \sup_{\phi \in \overline{L}_{k}\setminus \{0\}} \frac{\int_{M} \,\Vert \nabla \phi \Vert^{p} \,   dm  }{ \int_{M} \,\vert \phi \vert^{p} \,   dm }
                                            \leq \sup_{\phi \in L_{k+1}\setminus \{0\}} \frac{\int_{M} \,\Vert \nabla \phi \Vert^{p} \,   dm  }{ \int_{M} \,\vert \phi \vert^{p} \,   dm }.
\end{align*}
This implies the desired inequality.

We next prove $\widehat{\nu}_{k,p}(M,m)\leq 2^{p}\, \nu_{k,p}(M,m)$.
Let us fix a $k$-dimensional subspace $L_{k}$ of $W^{1,p}(M,m)$ with $L_{k}\cap \mathcal{C}=\{0\}$.
Lemma \ref{lem:Milman lemma} yields that
for every $\phi \in L_{k}$ we have
\begin{equation*}
\int_{M} \, \left \vert \phi-\int_{M}\,\phi\,dm \right\vert^{p} dm \leq 2^{p}\, \inf_{c\in \mathbb{R}} \int_{M}  \left \vert \phi-c  \right\vert^{p} dm.
\end{equation*}
Define a $(k+1)$-dimensional subspace $\widetilde{L}_{k+1}:=L_{k}\oplus \mathcal{C}$.
Then
\begin{align*}
\widehat{\nu}_{k,p}(M,m)&\leq \sup_{\phi \in \widetilde{L}_{k+1}\setminus \{0\}} \frac{\int_{M} \,\Vert \nabla \phi \Vert^{p}   dm  }{ \int_{M} \,\vert \phi \vert^{p} \,   dm }
                         \leq   \sup_{\phi \in L_{k}\setminus \{0\}} \frac{\int_{M} \,\Vert \nabla \phi \Vert^{p}   dm  }{\inf_{c\in \mathbb{R}} \int_{M} \,\vert \phi-c \vert^{p} \,   dm }\\
                        &\leq 2^{p} \sup_{\phi \in L_{k}\setminus \{0\}} \frac{\int_{M} \,\Vert \nabla \phi \Vert^{p}  dm  }{\int_{M} \, \left \vert \phi-\int_{M}\,\phi\,dm \right\vert^{p} dm }.
\end{align*}
We complete the proof.
\end{proof}

We further review the relation between $\nu_{k,p}(M,m),\,\widehat{\nu}_{k,p}(M,m)$,
and the spectrum of the \textit{weighted Laplacian}
\begin{equation*}
\Delta_{m}:=\Delta_{g}+g(\nabla f,\nabla \cdot),
\end{equation*}
where $\Delta_{g}$ is the Laplacian defined as the minus of the trace of Hessian.
We denote by
\begin{equation*}
0=\lambda_{0}(M,m)<\lambda_{1}(M,m) \leq \dots \leq \lambda_{k}(M,m)\leq \dots \nearrow +\infty
\end{equation*}
the all eigenvalues of $\Delta_{m}$,
counting multiplicity.
The min-max principle tells us that
\begin{equation}\label{eq:closed min max principle}
\lambda_{k}(M,m)=\nu_{k,2}(M,m)=\widehat{\nu}_{k,2}(M,m).
\end{equation}
For $p\in (1,\infty)$,
the \textit{weighted $p$-Laplacian $\Delta_{m,p}$} is defined as
\begin{equation*}
\Delta_{m,p}:=-e^{f}\,\Div_{g} \,\left(e^{-f} \Vert \nabla \cdot \Vert^{p-2}\, \nabla \cdot \right),
\end{equation*}
where $\Div_{g}$ is the divergence operator induced from $g$.
Here $\Delta_{m,2}=\Delta_{m}$.
Note that
$\Delta_{m,p}$ is non-linear in the case of $p\neq 2$ (for its spectral theory, see e.g., \cite{PAO} and the references therein).
Let $\lambda_{1,p}(M,m)$ stand for the smallest positive eigenvalue of $\Delta_{m,p}$,
which is known to be variationally characterized as follows (see e.g., Corollary 2.11 in \cite{H}):
\begin{equation*}
\lambda_{1,p}(M,m)=\inf_{\phi \in W^{1,p}(M,m)\setminus \mathcal{C}}   \frac{\int_{M} \,\Vert \nabla \phi \Vert^{p} \,   dm  }{ \inf_{c\in \mathbb{R}} \int_{M} \,\vert \phi-c \vert^{p} \,   dm }
\end{equation*}
for the set $\mathcal{C}$ of all constant functions on $M$.
In virtue of Lemma \ref{lem:Milman lemma} and Proposition \ref{prop:equivalence between Poincare constant and modified Poincare constant},
we see
\begin{equation*}
\lambda_{1,p}(M,m)^{\frac{1}{p}}\simeq \nu_{1,p}(M,m)^{\frac{1}{p}}\simeq \widehat{\nu}_{1,p}(M,m)^{\frac{1}{p}}.
\end{equation*}
Here $C_{1} \simeq C_{2}$ means that
$C_{1}$ and $C_{2}$ are equivalent up to universal explicit constants.
\begin{rem}\label{rem:eigenvalue of plaplacian}
Let us compare $\widehat{\nu}_{k,p}(M,m)$ with an eigenvalue of the $p$-Laplacian
 in the case of $p\neq 2$. In this case, a similar min-max principle to (\ref{eq:closed min max principle}) does not work in the higher-order case of $k\geq 2$.
We refer to \cite{PAO} for the summary of the higher-order eigenvalues of the weighted $p$-Laplacian.
We here recall some construction of the higher-order eigenvalues. Let $\mathcal{B}_{p}$ stand for the class of closed symmetric subsets of $\{\phi\in W^{1,p}(M,m) \mid \Vert \phi \Vert_{L^{p}}=1 \}$.
For $B \in \mathcal{B}_{p}$,
its \textit{genus} $\gamma^{+}(B)$ is defined as the supremum of $l\geq 1$ such that
there is an odd continuous surjective map from the $l$-dimensional unit sphere $\mathbb{S}^{l}$ in $\mathbb{R}^{l+1}$ to $B$.
Its \emph{cogenus} or \emph{Krasnosel'skii genus} $\gamma^{-}(B)$
 \cite{Kr} is also defined as the infimum of $l$ such that
there is an odd continuous map from $B$ to $\mathbb{S}^{l}$.
We now define
\begin{equation*}
\lambda^{\pm}_{k,p}(M,m):=\inf_{B\in \mathcal{B}^{\pm}_{k+1,p}} \sup_{\phi \in B} \int_{M}\Vert \nabla \phi \Vert^{p}dm,
\end{equation*}
where $\mathcal{B}^{\pm}_{k+1,p}:=\{B \in \mathcal{B}_{p} \mid \gamma^{\pm}(B)\geq k+1\}$.
Then the two sequences $\{\lambda^{\pm}_{k,p}(M,m)\}_{k}$ are known to
 be increasing and unbounded sequences of eigenvalues of the weighted
 $p$-Laplacian and it is known that
 $\lambda_{1,p}(M,m)$ coincides with $\lambda_{1,p}^{+}(M,m)$ (\cite{DR}). $\lambda_{k,p}^{+}(M,m)$ was introduced by Dr\'abek-Robinson
 \cite{DR}.
 For any $(k+1)$-dimensional subspace $L_{k+1}$ of $W^{1,p}(M,m)$,
 it holds that $\gamma^{\pm}(L_{k+1}\cap \{ \|\phi \|_{L^p}=1\})= k+1$ (see e.g., Proposition 5.2 in \cite{S} and Section 3 in \cite{AHP}),
 and hence
 \begin{align}\label{gamma1}
  \lambda^{\pm}_{k,p}(M,m)\leq \widehat{\nu}_{k,p}(M,m).
 \end{align}
 In particular, upper bounds for $\widehat{\nu}_{k,p}(M,m)$
 implies that for $\lambda^{\pm}_{k,p}(M,m)$.
 \end{rem}
 
%%%%%%%%%%%%%%%%%%%%%%%%
\subsection{Dirichlet Poincar\'e constants}\label{sec:Dirichlet Poincare constants}
Let $\bm$ be non-empty.
We next investigate the \textit{$k$-th Dirichlet $p$-Poincar\'e constant} defined as (\ref{eq:Dirichlet Poincare constant}).
Let
\begin{equation*}
0< \lambda^{D}_{1}(M,m)<\lambda^{D}_{2}(M,m) \leq \dots \leq \lambda^{D}_{k}(M,m)\leq \dots \nearrow +\infty
\end{equation*}
stand for the all Dirichlet eigenvalues of the weighted Laplacian $\Delta_{m}$,
counting multiplicity.
The min-max principle states
\begin{equation*}
\lambda^{D}_{k}(M,m)=\nu^{D}_{k,2}(M,m).
\end{equation*}
For $p\in (1,\infty)$,
let $\lambda^{D}_{1,p}(M,m)$ denote the smallest Dirichlet eigenvalue of the weighted $p$-Laplacian $\Delta_{m,p}$.
It is well-known that
$\lambda^{D}_{1,p}(M,m)$ is variationally characterized as
\begin{equation*}
\lambda^{D}_{1,p}(M,m)= \nu^{D}_{1,p}(M,m).
\end{equation*}

We now present the boundary concentration inequality,
which is a key ingredient of the proof of our main results.
For $A\subset M$,
let $B_{r}(A)$ denote its closed $r$-neighborhood in $M$.
\begin{prop}[\cite{FS}]\label{prop:boundary concentration inequality}
For every $r>0$ we have
\begin{equation}\label{eq:boundary concentration inequality}
m(M\setminus B_{r}(\bm)) \leq \exp \left(1-\nu^{D}_{1,p}(M,m)^{\frac{1}{p}}\,r \right).
\end{equation}
\end{prop}
\begin{proof}
The inequality (\ref{eq:boundary concentration inequality}) has been
 obtained in \cite{FS} in the unweighted case of $f=\log v_{g}(M)$ and $p=2$ (see Proposition 2.1 in \cite{FS}).
It can be proved by the same argument as in \cite{FS}.
We only outline the proof. 

We begin with showing that
\begin{equation}\label{eq:relative volume comparison and eigenvalue}
(1+\epsilon^{p}\,\nu^{D}_{1,p}(M,m))\,m(M\setminus B_{r+\epsilon}(\bm)) \leq m(M\setminus B_{r}(\bm))
\end{equation}
for all $\epsilon,r>0$.
Set $\Omega_{1}:=B_{r}(\bm),\,\Omega_{2}:=M\setminus B_{r+\epsilon}(\bm)$,
and $v_{\alpha}:=m(\Omega_{\alpha})$ for any $\alpha \in \{1,2\}$.
Define a Lipschitz function $\varphi:M\to \mathbb{R}$ by
\begin{equation*}
\varphi(x):=\min \left\{ \frac{1}{\epsilon} \,\,d(x,\Omega_{1}),\,\,1  \right\}.
\end{equation*}
It holds that
\begin{equation*}
\int_{M}\,\vert \varphi \vert^{p}\,dm \geq v_{2},\quad \int_{M}\,\Vert \nabla \varphi \Vert^{p}\,dm \leq \frac{1}{\epsilon^{p}}\,(1-v_{1}-v_{2}),
\end{equation*}
and hence
\begin{equation*}
\nu^{D}_{1,p}(M,m)\leq \frac{\int_{M}\, \Vert \nabla \varphi \Vert^{p}\,dm}{\int_{M}\, \vert \varphi \vert^{p}\,dm} \leq \frac{1}{\epsilon^{p}\,v_{2}}\,(1-v_{1}-v_{2}).
\end{equation*}
This proves (\ref{eq:relative volume comparison and eigenvalue}).

One can derive (\ref{eq:boundary concentration inequality}) from (\ref{eq:relative volume comparison and eigenvalue}).
For $\epsilon_{0}:=\nu^{D}_{1,p}(M,m)^{-\frac{1}{p}}$,
it suffices to consider the case where $r \in (0,\epsilon_{0})$.
Let $l$ be the integer determined by $\epsilon_{0}\,r^{-1} \in [(l+1)^{-1},l^{-1})$.
Applying (\ref{eq:relative volume comparison and eigenvalue}) to $m(M\setminus B_{l\,\epsilon_{0}}(\bm))$ iteratively,
we can conclude (\ref{eq:boundary concentration inequality}).
\end{proof}

%%%%%%%%%%%%%%%%%%%%%%%%%%%%
%%%%%%%%%%%%%%%%%%%%%%%%%%%%
%%%%%%%%%%%%%%%%%%%%%%%%%%%%
\section{Upper bounds for Poincar\'e constants}\label{sec:Poincare constants and concentration}
The present section is devoted to the proof of Theorem \ref{thm:main result}.
Throughout this section,
we always assume that
$M$ is closed.

%%%%%%%%%%%%%%%%%%%%%%%%
\subsection{Key principles}
In this subsection,
we prepare three key principles for the proof of Theorem \ref{thm:main result}.
For a domain $\Omega \subset M$,
we introduce the following local version of the Dirichlet Poincar\'e constant:
\begin{equation}\label{eq:local Dirichlet Poincare constant}
\nu^{D}_{k,p}(\Omega,m):=\inf_{L_{k}} \sup_{\phi \in L_{k}\setminus \{0\}} \frac{\int_{\Omega} \,\Vert \nabla \phi \Vert^{p} \,   dm  }{ \int_{\Omega} \,\vert \phi \vert^{p} \,   dm },
\end{equation}
where the infimum is taken over all $k$-dimensional subspaces $L_{k}$ of the $(1,p)$-Sobolev space $W^{1,p}_{0}(\Omega,m)$.
This is same as (\ref{eq:Dirichlet Poincare constant}) except for the difference of domain.

The following domain monotonicity principle follows from the above definition:
\begin{lem}\label{lem:domain monotonicity}
For a domain $\Omega \subset M$,
\begin{equation*}
\widehat{\nu}_{k,p}(M,m) \leq \nu^{D}_{k+1,p}(\Omega,m),
\end{equation*}
where $\widehat{\nu}_{k,p}(M,m)$ is defined as $(\ref{eq:modified Poincare constant})$.
\end{lem}

For a domain $\Omega \subset M$,
let $m_{\Omega}$ be the normalized volume measure on $\Omega$ defined as
\begin{equation}\label{eq:normalized measure}
m_{\Omega}:=\frac{1}{m(\Omega)}\,m|_{\Omega}.
\end{equation}

By the same argument as in the proof of Proposition \ref{prop:boundary concentration inequality},
we obtain the following boundary concentration inequality:
\begin{lem}\label{lem:local boundary concentration inequality}
Let $\Omega \subset M$ be a domain.
Then for all $r>0$,
\begin{equation*}
m_{\Omega}\left(\Omega  \setminus B_{r}(\partial \Omega)\right) \leq \exp\left(1-\nu^{D}_{1,p}(\Omega,m)^{\frac{1}{p}}\,r\right),
\end{equation*}
where $\partial \Omega$ denotes the boundary of $\Omega$.
\end{lem}

We further observe the following domain decomposing principle:
\begin{lem}\label{lem:higher order to smallest}
For any pairwise disjoint sequence $\{\Omega_{\alpha}\}^{k}_{\alpha=0}$ of domains in $M$,
we have
\begin{equation*}
\nu^{D}_{k+1,p}\left(  \bigsqcup^{k}_{\alpha=0} \Omega_{\alpha},m  \right)\leq \max_{\alpha=0,\dots,k} \nu^{D}_{1,p}(\Omega_{\alpha},m).
\end{equation*}
\end{lem}
\begin{proof}
Fix $\epsilon>0$.
For each $\alpha$,
take a non-zero $\phi_{\alpha}\in W^{1,p}_{0}(\Omega_{\alpha},m)$ with
\begin{equation}\label{eq:first approximation}
\frac{\int_{\Omega_{\alpha}} \,\Vert \nabla \phi_{\alpha} \Vert^{p} \,   dm}{\int_{\Omega_{\alpha}} \,\vert \phi_{\alpha} \vert^{p} \,   dm}<\nu^{D}_{1,p}(\Omega_{\alpha},m)+\epsilon.
\end{equation}
We define $\Omega:=\sqcup^{k}_{\alpha=0} \Omega_{\alpha}$,
and regard $\phi_{\alpha}$ as a function in $W^{1,p}_{0}(\Omega,m)$ by setting it as zero outside $\Omega_{\alpha}$.
We note that
$\{\phi_{\alpha}\}^{k}_{\alpha=0}$ are independent in $W^{1,p}_{0}(\Omega,m)$.
Let $L_{\epsilon}$ be the $(k+1)$-dimensional subspace of $W^{1,p}_{0}(\Omega,m)$ spanned by $\phi_{0},\dots,\phi_{k}$.
Fix $\delta>0$,
and take a non-zero $\phi \in L_{\epsilon}$ with
\begin{equation}\label{eq:second approximation}
\sup_{\psi \in L_{\epsilon}\setminus \{0\}} \frac{\int_{\Omega} \,\Vert \nabla \psi \Vert^{p} \,   dm  }{ \int_{\Omega} \,\vert \psi \vert^{p} \,   dm }<\frac{\int_{\Omega} \,\Vert \nabla \phi \Vert^{p} \,   dm  }{ \int_{\Omega} \,\vert \phi \vert^{p} \,   dm }+\delta.
\end{equation}
Let $c_{0},\dots,c_{k}$ be constants determined by $\phi=\sum^{k}_{\alpha=0}\,c_{\alpha} \phi_{\alpha}$,
and let $I$ be the set of all $\alpha$ such that $c_{\alpha}\neq 0$.
Since $\phi_{\alpha}\equiv 0$ outside $\Omega_{\alpha}$,
we possess
\begin{align}\label{eq:first lp representation}
\int_{\Omega} \,\vert \phi \vert^{p} \,   dm&=\sum_{\alpha\in I} \vert c_{\alpha} \vert^{p} \,\int_{\Omega_{\alpha}} \,\vert \phi_{\alpha} \vert^{p} \,   dm,\\ \label{eq:second lp representation}
\int_{\Omega} \,\Vert \nabla \phi \Vert^{p} \,   dm&=\sum_{\alpha \in I} \vert c_{\alpha} \vert^{p} \,\int_{\Omega_{\alpha}} \,\Vert \nabla \phi_{\alpha} \Vert^{p} \,   dm.
\end{align}

We now recall the following elementary inequality:
For $\beta=0,\dots l$,
and for constants $a_{\beta},b_{\beta}$ with $a_{\beta}\neq 0$,
we have
\begin{equation}\label{eq:elementary inequality}
\frac{\sum^{l}_{\beta=0}b_{\beta}}{\sum^{l}_{\beta=0}a_{\beta}}\leq \max_{\beta=0,\cdots,l} \frac{b_{\beta}}{a_{\beta}}.
\end{equation}
The inequality (\ref{eq:elementary inequality}) can be proven by induction on $l$.
From (\ref{eq:second approximation}), (\ref{eq:first lp representation}), (\ref{eq:second lp representation}), (\ref{eq:elementary inequality}), (\ref{eq:first approximation}),
we deduce
\begin{align*}
\nu^{D}_{k+1,p}\left( \Omega,m \right)&\leq \sup_{\psi \in L_{\epsilon}\setminus \{0\}} \frac{\int_{\Omega} \,\Vert \nabla \psi \Vert^{p} \,   dm  }{ \int_{\Omega} \,\vert \psi \vert^{p} \,   dm }
<\frac{\int_{\Omega} \,\Vert \nabla \phi \Vert^{p} \,   dm  }{ \int_{\Omega} \,\vert \phi \vert^{p} \,   dm }+\delta\\
&=\frac{\sum_{\alpha \in I} \vert c_{\alpha} \vert^{p} \,\int_{\Omega_{\alpha}} \,\Vert \nabla \phi_{\alpha} \Vert^{p} \,   dm}{\sum_{\alpha \in I} \vert c_{\alpha} \vert^{p} \,\int_{\Omega_{\alpha}} \,\vert \phi_{\alpha} \vert^{p} \,   dm}+\delta\\
&\leq \max_{\alpha \in I} \frac{\int_{\Omega_{\alpha}} \,\Vert \nabla \phi_{\alpha} \Vert^{p} \,   dm}{\int_{\Omega_{\alpha}} \,\vert \phi_{\alpha} \vert^{p} \,   dm}+\delta<\nu^{D}_{1,p}(\Omega_{\alpha_{0}},m)+\epsilon+\delta,
\end{align*}
where $\alpha_{0}\in I$ is determined by
\begin{equation*}
\frac{\int_{\Omega_{\alpha_{0}}} \,\Vert \nabla \phi_{\alpha_{0}} \Vert^{p} \,   dm}{\int_{\Omega_{\alpha_{0}}} \,\vert \phi_{\alpha_{0}} \vert^{p} \,   dm}=\max_{\alpha \in I} \frac{\int_{\Omega_{\alpha}} \,\Vert \nabla \phi_{\alpha} \Vert^{p} \,   dm}{\int_{\Omega_{\alpha}} \,\vert \phi_{\alpha} \vert^{p} \,   dm}.
\end{equation*}
We obtain
\begin{equation*}
\nu^{D}_{k+1,p}\left(\Omega,m \right)< \max_{\alpha=0,\dots,k}\nu^{D}_{1,p}(\Omega_{\alpha},m)+\epsilon+\delta.
\end{equation*}
Letting $\epsilon,\delta \to 0$ leads us to the desired conclusion.
\end{proof}

%%%%%%%%%%%%%%%%%%%%%%%%
\subsection{Proof of Theorem \ref{thm:main result}}

Theorem \ref{thm:main result} follows from the following:
\begin{thm}\label{thm:modified main result}
Let $(M,m)$ be a closed weighted Riemannian manifold.
For any sequence $\{A_{\alpha}\}^{k}_{\alpha=0}$ of Borel subsets of $M$,
we have
\begin{equation*}
 \widehat{\nu}_{k,p}(M,m)^{\frac{1}{p}} \leq \frac{2}{  \mathcal{D}(\{A_{\alpha}\})   }\max_{\alpha=0,\dots,k}  \log \frac{e(1-\sum_{\beta\neq \alpha}m(A_{\beta})    )}{m(A_{\alpha})}.
\end{equation*}
\end{thm}

\begin{proof}[Proof of Theorem \ref{thm:main result}]
In virtue of Theorem \ref{thm:modified main result} and Proposition \ref{prop:equivalence between Poincare constant and modified Poincare constant},
we get (\ref{eq:main inequality}).
This completes the proof.
\end{proof}
\begin{proof}[Proof of Theorem \ref{thm:modified main result}]
Set $R:=\mathcal{D}(\{A_{\alpha}\})$.
For a fixed $r\in (0,R/2)$,
we define a pairwise disjoint sequence $\{\Omega_{\alpha}\}^{k}_{\alpha=0}$ of domains $\Omega_{\alpha}\subset M$ as the open $r$-neighborhood of $A_{\alpha}$ in $M$.
Lemmas \ref{lem:domain monotonicity} and \ref{lem:higher order to smallest} lead us to
\begin{equation}\label{eq:domain monotonicity and higher order to smallest}
\widehat{\nu}_{k,p}(M,m)\leq \nu^{D}_{k+1,p}\left(\bigsqcup^{k}_{\alpha=0}\Omega_{\alpha},m\right)\leq \max_{\alpha}\nu^{D}_{1,p}(\Omega_{\alpha},m).
\end{equation}

For a fixed $\epsilon>0$,
we define
\begin{equation*}
r_{0}:=\frac{1}{     \nu^{D}_{1,p}(\Omega_{\alpha_{0}},m)^{\frac{1}{p}}}  \log \frac{e(1-\sum_{\alpha\neq \alpha_0}m(A_{\alpha})    )}{m(A_{\alpha_{0}})}+\epsilon,
\end{equation*}
where $\alpha_{0}$ is determined by
\begin{align*}
\nu^{D}_{1,p}(\Omega_{\alpha_{0}},m)=\max_{\alpha} \nu^{D}_{1,p}(\Omega_{\alpha},m).
\end{align*}
Notice that
\begin{equation}\label{eq:strict inequality}
\frac{m(A_{\alpha_{0}})}{1-\sum_{\alpha\neq \alpha_0}m(A_{\alpha})   }>\exp\left(1-\nu^{D}_{1,p}(\Omega_{\alpha_{0}},m)^{\frac{1}{p}}\,r_{0}\right).
\end{equation}
Let us prove $r\leq r_{0}$.
We have
\begin{equation*}
m(\Omega_{\alpha_{0}})\leq m\left(M\setminus \bigsqcup_{\alpha \neq \alpha_{0}}\Omega_{\alpha}\right)=1-\sum_{\alpha\neq \alpha_{0}}m(\Omega_{\alpha})\leq 1-\sum_{\alpha \neq \alpha_{0}}m(A_{\alpha}).
\end{equation*}
It follows that
\begin{equation}\label{eq:lower bound on largest set}
m_{\Omega_{\alpha_{0}}}(A_{\alpha_{0}})=\frac{m(A_{\alpha_{0}})}{m(\Omega_{\alpha_{0}})}\geq \frac{m(A_{\alpha_{0}})}{1-\sum_{\alpha \neq \alpha_{0}   } m(A_{\alpha}) }
.
\end{equation}
where $m_{\Omega_{\alpha_{0}}}$ is defined as (\ref{eq:normalized measure}).
By combining (\ref{eq:strict inequality}) and (\ref{eq:lower bound on largest set}),
and by Lemma \ref{lem:local boundary concentration inequality},
we obtain
\begin{equation*}
m_{\Omega_{\alpha_{0}}}(A_{\alpha_{0}})> \exp\left(1-\nu^{D}_{1,p}(\Omega_{\alpha_{0}},m)^{\frac{1}{p}}\,r_{0}\right) \geq m_{\Omega_{\alpha_{0}}}\left(\Omega_{\alpha_{0}}  \setminus B_{r_{0}}(\partial \Omega_{\alpha_{0}})\right).
\end{equation*}
This yields $B_{r_{0}}(\partial \Omega_{\alpha_{0}}) \cap A_{\alpha_{0}}\neq \emptyset$.
Therefore,
we see $r\leq r_{0}$.

By letting $\epsilon \to 0$ and $r\to R/2$,
and by (\ref{eq:domain monotonicity and higher order to smallest}),
we arrive at
\begin{align*}
R &\leq \frac{2}{     \nu^{D}_{1,p}(\Omega_{\alpha_{0}},m)^{\frac{1}{p}}}  \log \frac{e(1-\sum_{\alpha\neq \alpha_0}m(A_{\alpha})    )}{m(A_{\alpha_{0}})}\\
    &\leq \frac{2}{     \widehat{\nu}_{k,p}(M,m)^{\frac{1}{p}}}\max_{\alpha=0,\dots,k}  \log \frac{e(1-\sum_{\beta\neq \alpha}m(A_{\beta})    )}{m(A_{\alpha})}.
\end{align*}
This completes the proof of Theorem \ref{thm:modified main result}.
\end{proof}

Combining Theorem \ref{thm:modified main result} with (\ref{gamma1}) we obtain the
following corollary. Refer to Remark \ref{rem:eigenvalue of plaplacian} for
the definition of $\lambda_{k,p}^{\pm}(M,m)$.
\begin{cor}Under the same assumption of Theorem \ref{thm:modified main result}, we have
 \begin{align*}
  \lambda_{k,p}^{\pm}(M,m)^{\frac{1}{p}}\leq  \frac{2}{  \mathcal{D}(\{A_{\alpha}\})   }\max_{\alpha=0,\dots,k}  \log \frac{e(1-\sum_{\beta\neq \alpha}m(A_{\beta})    )}{m(A_{\alpha})}
  \end{align*}and in particular for $k=1$ we have
 \begin{align*}
  \lambda_{1,p}(M,m)^{\frac{1}{p}}\leq  \frac{2}{ d(A_0,A_1)
  }\max \Big\{  \log \frac{e(1- m(A_{1})    )}{m(A_{0})}, \log \frac{e(1- m(A_{0})    )}{m(A_{1})} \Big\}.
  \end{align*}
 \end{cor}

\begin{rem}\label{rem:comparing} We shall compare our estimate (\ref{eq:main inequality}) with the
 estimates obtained by Chung-Grigor'yan-Yau \cite{CGY1,CGY2} ((\ref{eq:CGY})) and
 Gozlan-Herry \cite{GH}.

 Let $\{A_{\alpha}\}_{\alpha=0}^k$ be a sequence of Borel subsets of $(M,m)$.
 By taking a permutation we may
 assume $m(A_0)\leq m(A_1)\leq \cdots \leq m(A_k)$. In this setting we
 have
 \begin{align}\label{eq:minmin}
 \log \frac{e(1-\sum_{\beta\neq 0}m(A_{\beta}))}{m(A_0)}= \max_{\alpha}\log \frac{e(1-\sum_{\beta\neq \alpha}m(A_{\beta}))}{m(A_{\alpha})}
  \end{align}and
 \begin{align*}
  \log \frac{e}{m(A_0)m(A_1)}= \max_{\alpha\neq \beta} \log \frac{e}{m(A_{\alpha})m(A_{\beta})}.
  \end{align*}Thus our estimate (\ref{eq:main inequality}) is better than
 Chung-Grigor'yan-Yau's estimate (\ref{eq:CGY}) as long as
 \begin{align*}
  e\,m(A_1)\Big(1-\sum_{\beta\neq 0}m(A_{\beta})\Big)^2\leq m(A_0).
  \end{align*}Otherwise Chung-Grigor'yan-Yau's estimate (\ref{eq:CGY})
 is better.

 In Proposition 2.2 of \cite{GH} Gozlan-Herry imposed the following assumption for $k$ Borel subsets $A_1,A_2,\dots,A_k$ of $M$;
 \begin{equation*}
  m(A_\alpha)+ \sum_{\beta=1}^k m(A_\beta)\geq 1 \text{ for any }\alpha=1,2,\cdots, k.
  \end{equation*}Under the assumption setting $A_0:=M\setminus B_r(\bigcup_{\alpha=1}^k
 A_{\alpha})$ they proved
 \begin{equation}\label{eneq:GH}
  \lambda_k(M,m)^{\frac{1}{2}}\leq \frac{2}{\mathcal{D}(\{A_{\alpha}\})} \phi \Big(\frac{1}{c} \log
   \frac{1-\sum_{\beta \neq 0} m(A_\beta)}{m(A_0)}\Big),
  \end{equation}where $\phi (x):=\max \{\sqrt{x},x\}$ and $c>0$ is a constant.
 In this setting observe that $m(A_0)\leq m(A_\alpha)$ for any $\alpha$
 and thus we have (\ref{eq:minmin}). Note also that
 $\phi^{-1}(x)=\min\{x,x^2\}$ and
 \begin{align*}
  \log \frac{e(1-\sum_{\beta \neq 0}m(A_{\beta}))}{m(A_0)}\geq 1.
  \end{align*}Hence as long as
 \begin{align*}
  \Big(\frac{1-\sum_{\beta\neq
  0}m(A_{\beta})}{m(A_0)}\Big)^{\frac{1}{c}-1}\geq e
  \end{align*}our estimate (\ref{eq:main inequality}) is better than
 (\ref{eneq:GH}) and otherwise (\ref{eneq:GH}) is better. We remark that
 Gozlan-Herry showed that $c=\frac{\log 5}{4}<1$ (proof of Theorem 2.1
 of \cite{GH}) and so $\frac{1}{c}-1>0$. Also our inequality holds
 without any restriction to $\{ A_{\alpha}\}$.
 \end{rem}
 
%%%%%%%%%%%%%%%%%%%%%%%%%%%%
%%%%%%%%%%%%%%%%%%%%%%%%%%%%
%%%%%%%%%%%%%%%%%%%%%%%%%%%%
\section{Upper bounds for Dirichlet Poincar\'e constants}\label{sec:Dirichlet Poincare constants and boundary concentration}
The aim of this section is to formulate an analogue of Theorem \ref{thm:main result} for Dirichlet eigenvalues.
In the present section,
we always assume that
$\partial M$ is non-empty.

We summarize some key lemmas to prove Theorem \ref{thm:main result with boundary}.
We denote by $\inte M$ the interior of $M$.
\begin{lem}\label{lem:Dirichlet domain monotonicity}
For a domain $\Omega \subset \inte M$,
we have
\begin{equation*}
\nu^{D}_{k,p}(M,m)\leq \nu^{D}_{k,p}(\Omega,m),
\end{equation*}
where $\nu^{D}_{k,p}(M,m)$ and $\nu^{D}_{k,p}(\Omega,m)$ are defined as $(\ref{eq:Dirichlet Poincare constant})$ and $(\ref{eq:local Dirichlet Poincare constant})$,
respectively.
\end{lem}

\begin{lem}\label{lem:Dirichlet local boundary concentration inequality}
Let $\Omega \subset \inte M$ be a domain.
For all $r>0$,
we have
\begin{equation*}
m_{\Omega}\left(\Omega  \setminus B_{r}(\partial \Omega)\right) \leq \exp\left(1-\nu^{D}_{1,p}(\Omega,m)^{\frac{1}{p}}\,r\right),
\end{equation*}
where $m_{\Omega}$ is defined as $(\ref{eq:normalized measure})$.
\end{lem}

\begin{lem}\label{lem:Dirichlet higher order to smallest}
For any pairwise disjoint sequence $\{\Omega_{\alpha}\}^{k}_{\alpha=1}$ of domains in $\inte M$,
we have
\begin{equation*}
\nu^{D}_{k,p}\left(  \bigsqcup^{k}_{\alpha=1} \Omega_{\alpha},m  \right)\leq \max_{\alpha=1,\dots,k} \nu^{D}_{1,p}(\Omega_{\alpha},m).
\end{equation*}
\end{lem}

One can verify Lemmas \ref{lem:Dirichlet domain monotonicity}, \ref{lem:Dirichlet local boundary concentration inequality} and \ref{lem:Dirichlet higher order to smallest}
by the same argument as in the proof of Lemmas \ref{lem:domain monotonicity}, \ref{lem:local boundary concentration inequality} and \ref{lem:higher order to smallest},
respectively.
We omit the proof.

\begin{proof}[Proof of Theorem \ref{thm:main result with boundary}]
We complete the proof by setting $R:=\mathcal{D}^{\partial}\bigl(\{A_{\alpha}\}\bigl)$,
and using Lemmas \ref{lem:Dirichlet domain monotonicity}, \ref{lem:Dirichlet local boundary concentration inequality} and \ref{lem:Dirichlet higher order to smallest}
instead of Lemmas \ref{lem:domain monotonicity}, \ref{lem:local boundary concentration inequality} and \ref{lem:higher order to smallest} along the line of the proof of Theorem \ref{thm:modified main result},
respectively.
\end{proof}

%%%%%%%%%%%%%%%%%%%%%%%%%%%%
%%%%%%%%%%%%%%%%%%%%%%%%%%%%
%%%%%%%%%%%%%%%%%%%%%%%%%%%%
\section{Upper bounds for multi-way isoperimetric constants}\label{sec:Multiway isoperimetric constants and concentration}
In this section,
we study multi-way isoperimetric constants which was introduced by
Miclo \cite{LM} and studied in \cite{DJM,F0,Liu,LGT,LM1}. Multi-way
isoperimetric constants are higher order version of isoperimetric
constants (Cheeger constants) and it is expected
that they possess similar properties of
eigenvalues of the Laplacian. 

%%%%%%%%%%%%%%%%%%%%%%%%
\subsection{Closed manifolds}\label{sec:isoperimetric Closed manifolds}
Let $M$ be closed.
For a Borel subset $A \subset M$,
\begin{equation*}
m^{+}(A):=\liminf_{r\to 0}\frac{m(U_{r}(A))-m(A)}{r},
\end{equation*}
where $U_{r}(A)$ is the open $r$-neighborhood of $A$.
The \textit{$k$-way isoperimetric constant} is defined as
\begin{equation*}
\mathcal{I}_{k}(M,m):=\inf_{\{A_{\alpha}\}} \max_{\alpha=0,\dots,k} \frac{m^{+}(A_{\alpha})}{m(A_{\alpha})},
\end{equation*}
where the infimum is taken over all pairwise disjoint sequences $\{A_{\alpha}\}^{k}_{\alpha=0}$ of non-empty Borel subsets $A_{\alpha}\subset M$.
When $k=1$,
this is nothing but the Cheeger constant.
The following relation formally established by Federer-Fleming \cite{FF} (cf. Theorem 9.6 in \cite{Li}, and Lemma 2.1):
\begin{equation}\label{eq:FFequiv}
\mathcal{I}_{1}(M,m) \simeq \nu_{1,1}(M,m).
\end{equation}

For a domain $\Omega \subset M$,
we consider the \textit{local $k$-way Dirichlet isoperimetric constant}
\begin{equation*}
\mathcal{I}^{D}_{k}(\Omega,m):=\inf_{\{A_{\alpha}\}} \max_{\alpha=1,\dots,k} \frac{m^{+}(A_{\alpha})}{m(A_{\alpha})},
\end{equation*}
where the infimum is taken over all pairwise disjoint sequences $\{A_{\alpha}\}^{k}_{\alpha=1}$ of non-empty Borel subsets of $\Omega$.
Due to Federer-Fleming \cite{FF},
we have the following (cf. Theorem 9.5 in \cite{Li}):
\begin{equation}\label{eq:Federer Fleming equation}
\mathcal{I}^{D}_{1}(\Omega,m)=\nu^{D}_{1,1}(\Omega,m),
\end{equation}
where the right hand side is defined as (\ref{eq:local Dirichlet Poincare constant}).

One can verify the following domain monotonicity principle:
\begin{lem}\label{lem:isoperimetric domain monotonicity}
For a domain $\Omega \subset M$,
we have
\begin{equation*}
\mathcal{I}_{k}(M,m) \leq \mathcal{I}^{D}_{k+1}(\Omega,m).
\end{equation*}
\end{lem}

The following boundary concentration inequality follows from Lemma \ref{lem:local boundary concentration inequality} with $p=1$,
and (\ref{eq:Federer Fleming equation}):
\begin{lem}\label{lem:isoperimetric local boundary concentration inequality}
Let $\Omega \subset M$ be a domain.
For all $r>0$,
we have
\begin{equation*}
m_{\Omega}\left(\Omega  \setminus B_{r}(\partial \Omega)\right) \leq \exp\left(1-\mathcal{I}^{D}_{1}(\Omega,m)\,r\right),
\end{equation*}
where $m_{\Omega}$ is defined as $(\ref{eq:normalized measure})$.
\end{lem}

By straightforward argument,
we also have the following:
\begin{lem}\label{lem:isoperimetric higher order to smallest}
For any pairwise disjoint sequence $\{\Omega_{\alpha}\}^{k}_{\alpha=0}$ of domains in $M$,
we have
\begin{equation*}
\mathcal{I}^{D}_{k+1}\left(  \bigsqcup^{k}_{\alpha=0} \Omega_{\alpha} ,m \right)\leq \max_{\alpha=0,\dots,k} \mathcal{I}^{D}_{1}(\Omega_{\alpha},m).
\end{equation*}
\end{lem}

We can show the following assertion by using Lemmas \ref{lem:isoperimetric domain monotonicity}, \ref{lem:isoperimetric local boundary concentration inequality} and \ref{lem:isoperimetric higher order to smallest}
along the line of the proof of Theorem \ref{thm:modified main result} instead of Lemmas \ref{lem:domain monotonicity}, \ref{lem:local boundary concentration inequality} and \ref{lem:higher order to smallest},
respectively.
The proof is left to the reader.
\begin{thm}\label{thm:isoperimetric main result}
Let $M=(M,m)$ be a closed weighted Riemannian manifold.
For any sequence $\{\Omega_{\alpha}\}^{k}_{\alpha=0}$ of Borel subsets of $M$,
we have
\begin{equation}\label{eq:isoperimetric main result}
\mathcal{I}_{k}(M,m)\leq \frac{2}{ \mathcal{D}(\{A_{\alpha}\})}\max_{\alpha=0,\dots,k}  \log \frac{e(1-\sum_{\beta\neq \alpha}m(A_{\beta})    )}{m(A_{\alpha})}.
\end{equation}
\end{thm}

\begin{rem}
The first author \cite{F0} stated the following inequality to (\ref{eq:isoperimetric main result}) (see Subsection 2.2 in \cite{F0}):
Under the same setting as in Theorem \ref{thm:isoperimetric main result},
it holds that
\begin{equation}\label{eq:LGT and CGY}
\mathcal{I}_{k}(M,m)\lesssim \frac{k^{3}}{   \mathcal{D}(\{A_{\alpha}\}) }  \max_{\alpha\neq \beta}  \log \frac{e}{m(A_{\alpha})m(A_{\beta})},
\end{equation}
here $C_{1} \lesssim C_{2}$ means that
$C_{1}\leq C\,C_{2}$ for some universal explicit constant $C>0$.
In \cite{F0},
he first pointed out that
the higher-order Cheeger inequality in the graph setting established by Lee-Gharan-Trevisan \cite{LGT} can be extended to the closed manifold setting by an appropriate modification of their proof (see Theorem 3.8 in \cite{LGT}, and Theorem 1.4 in \cite{F0}).
He concluded (\ref{eq:LGT and CGY}) by combining the higher-order
 Cheeger inequality with the inequality (\ref{eq:CGY}) of
 Chung-Grigor'yan-Yau \cite{CGY2}. Note that $k^3$ does not appear in
 (\ref{eq:isoperimetric main result}) and hence better than (\ref{eq:LGT and CGY}).
\end{rem}

%%%%%%%%%%%%%%%%%%%%%%%%
\subsection{Manifolds with boundary}\label{sec:isoperimetric manifolds with boundary}
We next consider the case where
$M$ is a compact manifold with boundary.
The \textit{$k$-way Dirichlet isoperimetric constant} is defined as
\begin{equation*}
\mathcal{I}^{D}_{k}(M,m):=\inf_{\{A_{\alpha}\}} \max_{\alpha=1,\dots,k} \frac{m^{+}(A_{\alpha})}{m(A_{\alpha})},
\end{equation*}
where the infimum is taken over all pairwise disjoint sequences $\{A_{\alpha}\}^{k}_{\alpha=1}$ of non-empty Borel subsets $A_{\alpha}\subset \inte M$.
Due to Federer-Fleming \cite{FF},
\begin{equation*}
\mathcal{I}^{D}_{1}(M,m)=\nu^{D}_{1,1}(M,m).
\end{equation*}

Similarly to the case where $M$ is closed,
we have the following:
\begin{lem}\label{lem:Dirichlet isoperimetric domain monotonicity}
For a domain $\Omega \subset \inte M$,
we have
\begin{equation*}
\mathcal{I}^{D}_{k}(M,m) \leq \mathcal{I}^{D}_{k}(\Omega,m).
\end{equation*}
\end{lem}

\begin{lem}\label{lem:Dirichlet isoperimetric local boundary concentration inequality}
Let $\Omega \subset \inte M$ be a domain.
For all $r>0$,
we have
\begin{equation*}
m_{\Omega}\left(\Omega  \setminus B_{r}(\partial \Omega)\right) \leq \exp\left(1-\mathcal{I}^{D}_{1}(\Omega,m)\,r\right).
\end{equation*}
\end{lem}

\begin{lem}\label{lem:Dirichlet isoperimetric higher order to smallest}
For any pairwise disjoint sequence $\{\Omega_{\alpha}\}^{k}_{\alpha=1}$ of domains in $\inte M$,
we have
\begin{equation*}
\mathcal{I}^{D}_{k}\left(  \bigsqcup^{k}_{\alpha=1} \Omega_{\alpha},m  \right)\leq \max_{\alpha=1,\dots,k} \mathcal{I}^{D}_{1}(\Omega_{\alpha},m).
\end{equation*}
\end{lem}

The above lemmas imply the following:
\begin{thm}\label{thm:Dirichlet isoperimetric main result}
Let $M=(M,m)$ be a compact weighted Riemannian manifold with boundary.
For any sequence $\{A_{\alpha}\}_{\alpha=1}^k$ of Borel subsets of $M$,
we have
\begin{equation*}
\mathcal{I}^{D}_{k}(M,m) \leq \frac{2}{ \mathcal{D}^{\partial}(\{ A_{\alpha}\})  }\max_{\alpha=1,\dots,k}  \log \frac{e(1-\sum_{\beta\neq \alpha}m(A_{\beta})    )}{m(A_{\alpha})}.
\end{equation*}
\end{thm}

%%%%%%%%%%%%%%%%%%%%%%%%%%%%
%%%%%%%%%%%%%%%%%%%%%%%%%%%%
%%%%%%%%%%%%%%%%%%%%%%%%%%%%
\section{Sharpness}\label{sec:Sharpness}

 Throughout this section,
 $M$ is assumed to be closed.

%%%%%%%%%%%%%%%%%%%%%%%%
\subsection{Sharpness of Theorem \ref{thm:main
 result}}\label{sec:Thm 1.1 sharp}
 In this subsection,
 we prove that
 Theorem \ref{thm:main result} is sharp with respect to the order of $k$.
 We first prepare the following proposition,
 which is an extension of Proposition 3.1 of
 \cite{F1} originally proved by Buser \cite{B} and Gromov \cite{G} independently.

 \begin{prop}\label{sect 6 prop}
  Let $\{B_{\alpha}\}_{\alpha=0}^{k-1}$ be a sequence of compact subsets of $M$ such that
  $M=\cup^{k-1}_{\alpha=0}B_{\alpha}$ and $m(B_{\alpha}\cap
  B_{\beta})=0$ for $\alpha \neq \beta$. Then we have
  \begin{equation*}
   \widehat{\nu}_{k,p}(M,m)^{\frac{1}{p}}\gtrsim \frac{1}{p}  \min_{\alpha} \mathcal{I}_{1}(B_{\alpha},m).
   \end{equation*}
   \end{prop}
  \begin{proof}We will follow the argument of Theorem 8.2.1 of \cite{B}.
   By the definition of $\widehat{\nu}_{k,p}(M,m)$, for a given $\epsilon>0$ there exists a $(k+1)$-dimensional subspace $L_{k+1}$ of
   $W^{1,p}(M,m)$ such that
   \begin{equation}\label{eq:prop6.1-1}
    \sup_{\phi\in L_{k+1}\setminus \{0\}}\frac{\int_M\|\nabla \phi \|^p
     dm}{\int_M|\phi|^pdm}<  \widehat{\nu}_{k,p}(M,m)+\epsilon.
    \end{equation}Since $L_{k+1}$ is $(k+1)$-dimensional, a standard
   argument of linear algebra implies the existence of
   $\phi_0\in L_{k+1}$ such that $\int_M |\phi_0|^p dm =1$ and $\phi_0 $
   is orthogonal to the characteristic funtion $1_{B_{\alpha}}$ of
   $B_{\alpha}$ for $\alpha=0,\dots, k-1$, i.e., $\int_{B_{\alpha}}\phi_0\, dm=0$.
   By the definition of ${\nu}_{1,p}(B_{\alpha},m)$ we get
   \begin{equation*}
    {\nu}_{1,p}(B_{\alpha},m)\int_{B_{\alpha}}|\phi_0|^p dm \leq
     \int_{B_{\alpha}} \|\nabla \phi_0\|^p dm.
    \end{equation*}From $\int_M |\phi_0|^p dm =1$ it follows that
   \begin{align*}
    \min_{\alpha}{\nu}_{1,p}(B_{\alpha},m)\leq & \ 
     \sum_{\alpha=0}^{k-1}{\nu}_{1,p}(B_{\alpha},m)\int_{B_{\alpha}}|\phi_0|^p
     dm\\
    \leq & \  \sum_{\alpha=0}^{k-1}\int_{B_{\alpha}} \|\nabla \phi_0\|^p
     dm
    =   \ \int_M \| \nabla \phi_0\|^p dm.
    \end{align*}Combining this with (\ref{eq:prop6.1-1}) and letting $\epsilon \to 0$
   yield
   \begin{equation}\label{eq:prop6.1-2}
    \min_{\alpha}{\nu}_{1,p}(B_{\alpha},m)\leq \widehat{\nu}_{k,p}(M,m).
    \end{equation}
    
    On the other hand,
    Proposition 2.5 in \cite{Mi1} tells us that
    \begin{equation}\label{eq:prop6.1-3}
    {\nu}_{1,p}(B_{\alpha},m)^{\frac{1}{p}}\gtrsim
   \frac{1}{p}{\nu}_{1,1}(B_{\alpha},m).
    \end{equation}
    Therefore,
    (\ref{eq:prop6.1-2}), (\ref{eq:prop6.1-3}), and the equivalence of $\nu_{1,1}(B_{\alpha},m)$ and $\mathcal{I}_{1}(B_{\alpha},m)$ lead us to the desired inequality. 
   This completes the proof.
   \end{proof}

  Let us show the sharpness of Theorem \ref{thm:main result} with
  respect to the order of $k$.
  We will show it by constructing an example,
  and estimating its Poincar\'e constant.
  We can construct an example of closed manifold,
  but its formulation will be slightly complicated.
  To avoid complexity,
  we will construct an example of compact manifold with boundary on which we consider the Neumann boundary condition,
  and conclude the sharpness for it.
  Here we notice that
  Theorem \ref{thm:main result} and other results on closed manifolds can be applied to the Neumann eigenvalues on compact manifolds with boundary due to the min-max principle.
  In the end of this section,
  we will sketch an idea of the construction of the example of closed manifold (see Remark \ref{rem:closed ex}).
  
  For $a\in (0,1)$,
we work on a rectangle $[0,1]\times [0,a]$ in $\mathbb{R}^{2}$.
Given $s,t\geq 0$ setting 
\begin{equation*}
\Omega_{s,t}:=\left[s,s+\frac{1}{2(k+1)}   \right]\times [0,t],
\end{equation*}we define
\begin{equation*}
M_{a}:=\bigcup_{\alpha=0,\dots,k}\left( \Omega_{
                                  \frac{2\alpha}{2(k+1)},a} \cup
                                  \Omega_{\frac{2\alpha+1}{2(k+1)},\frac{a}{k+1}} \right).
\end{equation*}
Let us consider a pairwise disjoint sequence
 $\{A_{\alpha}\}^{k+1}_{\alpha=0}$ of subsets $A_{\alpha} \subset
 M_{a}$ defined as follows: For $\alpha=0,1,\dots, k-1$ we set
 \begin{align*}
  A_{\alpha}:=\Omega_{\frac{2\alpha}{2(k+1)},a}.
  \end{align*}For $\alpha =k, k+1$ let us set
 \begin{align*}
  A_k := \left[ \frac{6k}{6(k+1)},\frac{6k+1}{6(k+1)}\right]\times[0,a]
  \end{align*}and
 \begin{align*}
 A_{k+1}:= \left[\frac{6k+2}{6(k+1)},\frac{6k+3}{6(k+1)}\right]\times[0,a].
  \end{align*}
If $d_{M_{a}}$ and $m_{M_{a}}$ denote the Riemannian distance and the
 normalized uniform volume measure on $M_{a}$ respectively, then we see
\begin{equation*}
m_{M_{a}}\left(A_{\alpha}  \right)\geq \frac{1}{12(k+1)}, \ 
 d_{M_{a}}(A_{\alpha},A_{\beta})\geq \frac{1}{6(k+1)}
 \end{equation*}and
 \begin{equation*}
 \sum_{\beta\neq  \alpha} m_{M_a}(A_{\beta})\geq 1-\frac{2}{k+1}=\frac{k-1}{k+1}.
\end{equation*}
Hence, Theorem \ref{thm:main result} leads to
\begin{equation}\label{eq:upper example}
\nu_{k+1,p}(M_{a},m_{M_a})^{\frac{1}{p}}\lesssim k+1 \sim k.
\end{equation}
On the other hand, one can show
 \begin{equation}\label{eq:lower example}
  \nu_{k+1,p}(M_a,m_{M_a})^{\frac{1}{p}} \gtrsim \frac{k}{p} 
  \end{equation}provided that $a\sim \frac{1}{k+1}$ as follows: For
 $\alpha=0,1,\dots,k$ we set
 \begin{equation*}
  B_{\alpha}:= \Omega_{\frac{2\alpha}{2(k+1)},a} \cup
                                  \Omega_{\frac{2\alpha+1}{2(k+1)},\frac{a}{k+1}}. 
  \end{equation*}Then we have $M_a=\cup^{k}_{\alpha=0}B_{\alpha}$ and
                                  $m_{M_a}(B_{\alpha}\cap B_{\beta})=0$
                                  for $\alpha \neq \beta$. Note that $\mathcal{I}_1(s
 B_{\alpha})=s^{-1}\mathcal{I}_1(B_{\alpha})$ where $sB_{\alpha}:=\{ s
 x\mid x\in B_{\alpha}\}$. Here we consider the Cheeger constant $\mathcal{I}_1(s
 B_{\alpha})$ in $s
 M_a$ with the normalized uniform volume measure and the Cheeger constant $\mathcal{I}_1(
 B_{\alpha})$ in $(M_a,m_{M_a})$. Since $a\sim \frac{1}{k+1}$ we have $\mathcal{I}_{1}(B_{\alpha})\gtrsim
  k $. Propositions \ref{prop:equivalence between Poincare constant and modified Poincare constant} and \ref{sect 6 prop} yield (\ref{eq:lower example}).
Comparing (\ref{eq:upper example}) with (\ref{eq:lower example}),
we can conclude that
Theorem \ref{thm:main result} is sharp with respect to the order of $k$.

\begin{rem}
 This example does not deny the possibility of
 \begin{equation*}
  \lambda_k(M,m)^{\frac{1}{2}}\lesssim
   \frac{1}{\mathcal{D}(\{A_{\alpha}\})\log k}\max_{\alpha }\log \frac{1}{m(A_{\alpha})}.
  \end{equation*}This inequality was conjectured in \cite{F1} under the
 assumption of the non-negativity of Ricci curvature. One might not need
 the assumption.
 \end{rem}
 
 %%%%%%%%%%%%%%%%%%%%%%%%
 \subsection{Note on Theorem \ref{thm:isoperimetric main result}}\label{sec:Thm 5.4 sharp}
 Unlike Theorem \ref{thm:main result},
 the authors do not know whether Theorem \ref{thm:isoperimetric main result} is sharp with respect to the order
 of $k$.
 In this subsection,
 we discuss its possibility.
 
 For a measurable function $\phi:M\to \mathbb{R}$,
 a real number $\med_\phi$ is called its \emph{median} if it satisfies 
\begin{equation*}
 m(\{x\in M \mid \phi(x)\geq \med_{\phi}\})\geq \frac{1}{2},\quad m(\{x\in M \mid \phi(x)\leq \med_{\phi}\})\geq \frac{1}{2}.
\end{equation*}
Before we start the discussion,
let us recall the following fact due to Maz'ya \cite{mazya}
 and Federer-Fleming \cite{FF} (see e.g., Lemma 2.2 in \cite{Mi1}):
  \begin{thm}\label{11poin}The Cheeger constant $\mathcal{I}_{1}(M,m)$ is the best constant for the following Poincar\'e inequality: For any $\phi \in W^{1,1}(M,m)$,
   \begin{equation*}
   \mathcal{I}_{1}(M,m) \int_M \vert \phi-\med_{\phi} \vert dm \leq  \int_M \Vert \nabla \phi \Vert dm. 
    \end{equation*}
   \end{thm}
    
   Instead of the $k$-way isoperimetric constant $\mathcal{I}_{k}(M,m)$,
   we consider the \textit{modified $k$-way isoperimetric constant $\widehat{\mathcal{I}}_{k}(M,m)$} defined by
    \begin{equation*}
    \widehat{\mathcal{I}}_{k}(M,m):=\inf_{\{A_{\alpha}\}} \max_{\alpha=0,\dots,k} \frac{m^{+}(A_{\alpha})}{m(A_{\alpha})},
    \end{equation*}
where the infimum is taken over all pairwise disjoint sequences $\{A_{\alpha}\}^{k}_{\alpha=0}$ of non-empty Borel subsets with $M=\sqcup^{k}_{\alpha=0}A_{\alpha}$.
We see that $\mathcal{I}_{k}(M,m)\leq \widehat{\mathcal{I}}_{k}(M,m)$;
moreover the equality holds when $k=1$, and $\widehat{\mathcal{I}}_{1}(M,m)$ is equivalent to $\nu_{1,1}(M,m)$ (cf. (\ref{eq:FFequiv})).
We have the following assertion:

   \begin{prop}\label{sect 6 prop2}Let $\{B_{\alpha}\}_{\alpha=0}^{k-1}$ be a sequence of compact subsets of $M$ such that
  $M=\cup^{k-1}_{\alpha=0} B_{\alpha}$ and $m(B_{\alpha}\cap
  B_{\beta})=0$ for $\alpha \neq \beta$. Then we have
  \begin{equation*}
   \widehat{\mathcal{I}}_{k}(M,m)\geq  \min_{\alpha} \mathcal{I}_{1}(B_{\alpha},m).
   \end{equation*}
   \end{prop}
    \begin{proof}For any $\epsilon>0$ there exists a pairwise disjoint sequences $\{A_{\beta}\}^{k}_{\beta=0}$ of non-empty Borel subsets with $M=\sqcup^{k}_{\beta=0}A_{\beta}$ such that
  \begin{align}\label{transa}
   \widehat{\mathcal{I}}_{k}(M,m)+\epsilon> \max_{\beta}\frac{m^{+}(A_{\beta})}{m(A_{\beta})}.
   \end{align}Using the Borsuk-Ulam theorem we can get constants
  $c_0,c_1,\dots,c_{k}$ such that $\phi_{0}:= \sum_{\beta=0}^{k}c_{\beta} 1_{A_{\beta}}$
  bisects each $B_0, B_1,\dots B_{k-1}$, i.e.,
     \begin{align*} m(B_{\alpha}\cap
   \phi_{0}^{-1}[\, 0,\infty\,))\geq \frac{m(B_{\alpha})}{2},\quad m(B_{\alpha}\cap
    \phi_{0}^{-1}(\,-\infty,0\,])\geq \frac{m(B_{\alpha})}{2}.
      \end{align*}In fact, according to Corollary of \cite{stonetukey}, in
   order to
   bisect $k$ subsets by a finite combination of $1_{A_0},
   1_{A_1},\dots,1_{A_k}$, it suffices to check that $1_{A_0},1_{A_1},\dots, 1_{A_k}$ are
   linearly independent modulo sets of measure zero (i.e., whenever
   $a_01_{A_0}+a_1 1_{A_1}+\cdots +a_{k}1_{A_{k}}=0$ over a Borel subset of positive
   measure, we have $a_0=a_1=\cdots =a_k=0$). This holds since $M=\sqcup^{k}_{\beta=0}A_{\beta}$.
     We now apply Theorem \ref{11poin} to Lipschitz functions which
     approximate, the characteristic function $1_{A_{\beta}}$, in an
     appropriate sense.
     Thus we see
     \begin{align*}
      \min_{\alpha} \mathcal{I}_{1}(B_{\alpha},m) \int_M \vert \phi_{0} \vert \,dm\leq \ &
      \sum_{\alpha=0}^{k-1}\mathcal{I}_{1}(B_{\alpha},m)
      \int_{B_{\alpha}}\vert  \phi_{0} \vert \,dm\\
      \leq \ &
      \sum_{\alpha=0}^{k-1}\int_{B_{\alpha}}\Vert \nabla \phi_{0} \Vert dm
      = \int_M \Vert \nabla \phi_{0}\Vert \,dm\\
      \leq \ & \sum_{\beta=0}^{k}|c_{\beta}|m^{+}(A_{\beta}).
      \end{align*}From
     \begin{equation*}
      \int_M \vert \phi_{0}\vert\,dm = \sum_{\beta=0}^k |c_{\beta}|m(A_{\beta}),
      \end{equation*}we derive
     \begin{equation*}
      \min_{\alpha} \mathcal{I}_{1}(B_{\alpha},m)\leq \max_{\beta}\frac{m^+(A_{\beta})}{m(A_{\beta})}.
      \end{equation*}Therefore,
      this together with (\ref{transa}) completes the proof.
    \end{proof}
         
    We now consider $M_{a}$ introduced in the above subsection.
    By Theorem \ref{thm:isoperimetric main result} we possess $\mathcal{I}_{k+1}(M_a,m_{M_a})\lesssim k+1$.
    On the other hand,
    Proposition \ref{sect 6 prop2}, and the same argument as in the above subsection yield $\widehat{\mathcal{I}}_{k+1}(M_a,m_{M_a})\gtrsim k$ provided that $a\sim \frac{1}{k+1}$.
    If one can replace $\widehat{\mathcal{I}}_{k+1}(M,m)$ with $\mathcal{I}_{k+1}(M,m)$ in Proposition \ref{sect 6 prop2},
    then we can show that $\mathcal{I}_{k+1}(M_a,m_{M_a})\gtrsim k$,
    and conclude that Theorem \ref{thm:isoperimetric main result} is also sharp.
    
    \begin{rem}\label{rem:closed ex}
    We can discuss the sharpness of Theorems \ref{thm:main result} and \ref{thm:isoperimetric main result} by considering the following closed manifold instead of $M_{a}$:
    We construct $(k+1)$ flat tori by identifying the edges of large rectangles in $M_{a}$,
    and connect them by cylinders constructed by identifying the upper edges of small rectangles in $M_{a}$ with their lower edges.
    We can discuss the sharpness by applying the same argument to this closed manifold.
    \end{rem}

%%%%%%%%%%%%%%%%%%%%%%%%%%%%
%%%%%%%%%%%%%%%%%%%%%%%%%%%%
%%%%%%%%%%%%%%%%%%%%%%%%%%%%
\section{Poincar\'e constants and inscribed radii}\label{sec:Inscribed radius estimates}
In this section,
we always assume that
$M$ is a compact manifold with boundary.
Its \textit{inscribed radius} is defined as
\begin{equation*}
\IR M:=\sup_{x\in M} d(x,\bm).
\end{equation*}
We will discuss upper bounds for the inscribed radii under non-negativity of the weighted Ricci curvature.

We denote by $\dim M$ the dimension of $M$,
and by $\ric_{g}$ the Ricci curvature induced from the Riemannian metric $g$.
For $N\in (-\infty,\infty]$,
the $N$-\textit{weighted Ricci curvature} is defined as follows (\cite{BE}, \cite{Lic}):
\begin{equation*}
\ric^{N}_{m}:=\ric_{g}+\Hess f-\frac{d f \otimes d f}{N-\dim M},
\end{equation*}
where $d f$ and $\Hess f$ are the differential and the Hessian of $f$,
respectively.
Let $\ric^{N}_{m,M}$ stand for the infimum of $\ric^{N}_{m}$ over the unit tangent bundle over $M$.
We produce the following upper bound of the inscribed radius based on the fundamental principles established in Section \ref{sec:Dirichlet Poincare constants and boundary concentration}.
\begin{prop}\label{prop:inscribed estimate}
For $N\in [\dim M,\infty)$
we assume $\ric^{N}_{m,M}\geq 0$.
Then
\begin{equation}\label{eq:inrad est}
\IR M \leq \frac{2}{\nu^{D}_{1,p}(M,m)^{\frac{1}{p}}}\left( 1+N \log 2  \right).
\end{equation}
\end{prop}
\begin{proof}
We set $R:=\IR M$.
Since $M$ is compact,
we can find a point $x_{0}\in M$ such that $R=d(x_{0},\bm)$.
For $r>0$,
let $\Omega_{r}$ denote the the open ball of radius $r$ centered at $x_{0}$.
By the definition of the inscribed radius,
$\Omega_{R}$ is contained in $\inte M$.
This enables us to apply the volume comparison theorem of Bishop-Gromov type under $\ric^{N}_{m}\geq 0$ obtained by Qian \cite{Q} to $\Omega_{R},\,\Omega_{R/2}$ (see Corollary 2 in \cite{Q}).
Hence
\begin{equation}\label{eq:BishopGromov}
2^{-N}\leq \frac{m(\Omega_{R/2})}{m(\Omega_{R})}=m_{\Omega_{R}}\left(\Omega_{R/2}\right),
\end{equation}
where $m_{\Omega_{R}}$ is defined as $(\ref{eq:normalized measure})$.

On the other hand,
from Lemmas \ref{lem:Dirichlet domain monotonicity} and \ref{lem:Dirichlet local boundary concentration inequality}
we deduce
\begin{align}\label{eq:reverse BishopGromov}
m_{\Omega_{R}}\left(\Omega_{R}  \setminus B_{r}(\partial
 \Omega_{R})\right) &\leq \exp\left(1-\nu^{D}_{1,p}(\Omega_{R},m)^{\frac{1}{p}}\,r\right)\\
                                                                                                                      &\leq \exp\left(1-\nu^{D}_{1,p}(M,m)^{\frac{1}{p}}\,r\right) \notag
\end{align}
for all $r>0$.
By letting $r\to R/2$ in (\ref{eq:reverse BishopGromov}),
and combining it with (\ref{eq:BishopGromov}),
we arrive at
\begin{equation*}
2^{-N}\leq m_{\Omega_{R}}\left(\Omega_{R/2}\right)\leq  \exp\left(1-\frac{1}{2}\nu^{D}_{1,p}(M,m)^{\frac{1}{p}}R\right).
\end{equation*}
This yields the desired inequality.
\end{proof}

\begin{rem}
The authors wonder whether
one can extend Proposition \ref{prop:inscribed estimate} to the higher-order case of $k\geq 2$ as follows:
Under the same setting as in Proposition \ref{prop:inscribed estimate},
\begin{equation*}
\IR M \leq \frac{2k}{\nu^{D}_{k,p}(M,m)^{\frac{1}{p}}}\left( 1+N \log 2  \right).
\end{equation*}
In the case of $p=2$,
a similar inequality was implicitly shown in \cite{S}.
The second author \cite{S} has remarked that
a classical method by Cheng \cite{C} leads us to the following (see the inequality (2.16) in Remark 2.4 in \cite{S}):
If $\ric^{N}_{m,M}\geq 0$ for $N\in [\dim M,\infty)$,
then
\begin{equation}\label{eq:Cheng}
\nu^{D}_{k,2}(M,m)\leq 2N(N+4)k^{2} (\IR M)^{-2}.
\end{equation}
The second author \cite{S} has only considered the unweighted case where $f=0$ and $N=\dim M$,
but one can extend it to the weighted setting by using weighted comparison geometric results developed by Qian \cite{Q}.
The inequality (\ref{eq:Cheng}) can be rewritten as
\begin{equation*}
\IR M\leq \frac{\sqrt{2}k}{\nu^{D}_{k,2}(M,m)^{\frac{1}{2}}}\sqrt{N(N+4)}.
\end{equation*}
\end{rem}

\begin{rem}
The authors do not know whether
Proposition \ref{prop:inscribed estimate} is optimal over $p$ and $N$.
For example,
for the Euclidean $r$-ball $B_{r}$,
its first Dirichlet eigenvalue $\lambda^{D}_{1,p}(B_{r})$ of the (unweighted) $p$-Laplacian satisfies
\begin{equation*}\label{eq:euc ball}
\IR B_{r}=r \geq \frac{\dim B_{r}}{p\,\lambda^{D}_{1,p}(B_{r})^{\frac{1}{p}}},
\end{equation*}
which is a consequence of the $p$-Cheeger inequality (see e.g., Theorem 2 in \cite{T}, Theorem 4.1 in \cite{Ma}, Theorem 4.3 in \cite{M}).
In particular,
this simplest example does not show that
Proposition \ref{prop:inscribed estimate} is optimal with respect to $p$.
Also,
Mao \cite{M} has obtained a more refined estimate than (\ref{eq:inrad est}) in the case where
$M$ is a closed ball in a Riemannian manifold of non-negative Ricci curvature (see Theorem 4.3 in \cite{M}).
One might be able to extend his estimate to our general setting.
\end{rem}

%%%%%%%%%%%%%%%%%%%%%%%%%%%%
%%%%%%%%%%%%%%%%%%%%%%%%%%%%
%%%%%%%%%%%%%%%%%%%%%%%%%%%%
\section{Discrete cases}\label{sec:Discrete cases}
The aim of this last section is to point out that
we possess an discrete analogue of Theorem \ref{thm:main result}.
We only state the setting,
and the statement of the key lemmas and the main results.
One can show them along the line of the argument in Section \ref{sec:Poincare constants and concentration}.
The proof is left to the reader.

%%%%%%%%%%%%%%%%%%%%%%%%
\subsection{Weighted graphs and their Poincar\'{e} constants}
We explain our discrete setting.
We refer to Chapter 1 in \cite{Gr} (see also Theorem 3.3 in \cite{L}, and Section 4 in \cite{GH}):
 Let $G=(V,E)$ be a simple connected finite graph. A \emph{weighted
 graph} is a pair $(G,\mu)$ where $\mu$ is a nonnegative function on
 $V\times V$ such that (1) $\mu_{x,y}=\mu_{y,x}$; and (2) $\mu_{x,y}>0$ iff $x$
 and $y$ form an edge. We consider the graph distance on $V$.

 For a point $x\in V$ and a subset $\Omega \subset V$ we set
 $\mu(x):=\sum_{y\in V}\mu_{x,y}$ and $\mu(\Omega):=\sum_{x\in \Omega}\mu(x)$.
 For a function $\phi:V\to \mathbb{R}$ we
 define
 \begin{align*}
  \overline{\phi}:=\frac{1}{\mu(V)}\sum_{x\in V}\phi(x)\mu(x), 
 \end{align*}i.e., $\overline{\phi}$ is the mean of $\phi$ with respect to the
 weight $\mu$. Given $0\leq k \leq \#V-1$ we introduce the \textit{$k$-th discrete $p$-Poincar\'e
 constant} of $(G,\mu)$ as
\begin{equation*}
\nu_{k,p}(G,\mu):=\inf_{L_{k}} \sup_{\phi \in L_{k}\setminus \{0\}} \frac{  \sum_{x,y\in V} \vert \phi(y)-\phi(x) \vert^{p}\,\mu_{x,y}    }{2\sum_{x\in V}\vert \phi(x)-\overline{\phi} \vert^{p}\, \mu(x)    },
\end{equation*}
where the infimum is taken over all $k$-dimensional subspace $L_k$ of the
associated $L^{p}$-space $L^{p}(G,\mu)$ on $(G,\mu)$ which
does not contain nontrivial constant functions.
The \textit{(weighted) graph Laplacian $\Delta_{\mu}$} is defined as
\begin{align*}
 \Delta_{\mu}\phi(x):=\phi(x)-\sum_{y\in V}\phi(y)\frac{\mu_{x,y}}{\mu(x)}.
 \end{align*}Then the $k$-th nontrivial eigenvalue $\lambda_k(G,\mu)$ of
 $\Delta_{\mu}$ coincides with $\nu_{k,2}(G,\mu)$.

 As in the Riemannian case we compare discrete
 Poincar\'e constants with its modified version and then we will work
 with the modified version.
  We introduce the \textit{$k$-th discrete modified $p$-Poincar\'e
 constant} of $(G,\mu)$ as
 \begin{equation*}
 \widehat{\nu}_{k,p}(G,\mu):=\inf_{L_{k+1}} \sup_{\phi \in L_{k+1}\setminus \{0\}} \frac{  \sum_{x,y\in V} \vert \phi(y)-\phi(x) \vert^{p}\,\mu_{x,y}    }{2\sum_{x\in V}\vert \phi(x) \vert^{p}\, \mu(x)    },
\end{equation*}where the infimum is taken over all $(k+1)$-dimensional subspace $L_{k+1}$ of $L^{p}(G,\mu)$. Then we have $\nu_{k,2}(G,\mu)=
\widehat{\nu}_{k,2}(G,\mu)$ and the following:
\begin{lem}
 \begin{align*}
\nu_{k,p}(G,\mu)\leq \widehat{\nu}_{k,p}(G,\mu)\leq 2^p\nu_{k,p}(G,\mu).
  \end{align*}
 \end{lem}

To analyze discrete modified Poincar\'e constants we prepare its local version.
For $\Omega \subset V$,
we denote by $\partial_{v} \Omega$ its vertex boundary defined as the set of all $x \in \Omega$ satisfying $\mu_{x,y}>0$ for some $y\in V\setminus \Omega$.
Further,
let $L^{p}_{0}(\Omega,\mu)$ denote a subspace of $L^{p}(\Omega,\mu)$ defined as the set of all functions $\phi \in L^{p}(\Omega,\mu)$ with $\phi|_{\partial_{v} \Omega}=0$.
We now define the \textit{$k$-th discrete Dirichlet $p$-Poincar\'e constant} by
\begin{equation*}
\nu^{D}_{k,p}(\Omega,\mu):=\inf_{L_{k,0}} \sup_{\phi \in L_{k,0}\setminus \{0\}} \frac{  \sum_{x,y\in \Omega} \vert \phi(y)-\phi(x) \vert^{p}\,\mu_{x,y}    }{2\sum_{x\in \Omega}\vert \phi(x) \vert^{p}\, \mu(x)    },
\end{equation*}
where the infimum is taken over all $k$-dimensional subspaces $L_{k,0}$ of $L^{p}_{0}(\Omega,\mu)$ .
We summarize key tools for the proof of our main result.
\begin{lem}\label{lem:discrete domain monotonicity}
For $\Omega \subset V$,
we have
\begin{equation*}
\widehat{\nu}_{k,p}(G,\mu) \leq \nu^{D}_{k+1,p}(\Omega,\mu).
\end{equation*}
\end{lem}

\begin{lem}\label{lem:discrete local boundary concentration inequality}
Let $\Omega \subset V$.
For all $r>0$,
we have
\begin{equation*}
\frac{\mu\left(\Omega  \setminus B_{r}(\partial_{v} \Omega)\right)}{\mu(\Omega)} \leq \exp\left(1-\nu^{D}_{1,p}(\Omega,\mu)^{\frac{1}{p}}\,r\right).
\end{equation*}
\end{lem}

\begin{lem}\label{lem:discrete higher order to smallest}
For any pairwise disjoint sequence $\{\Omega_{\alpha}\}^{k}_{\alpha=0}$ of subsets in $V$,
we have
\begin{equation*}
\nu^{D}_{k+1,p}\left(  \bigsqcup^{k}_{\alpha=0} \Omega_{\alpha},\mu  \right)\leq \max_{\alpha=0,\dots,k} \nu^{D}_{1,p}(\Omega_{\alpha},\mu).
\end{equation*}
\end{lem}

We can conclude the following discrete analogue of Theorem \ref{thm:main result}:
\begin{thm}\label{thm:discrete modified main result}
 Let $(G,\mu)$ be a weighted graph. For any sequence
 $\{A_{\alpha}\}^{k}_{\alpha=0}$ of subsets of $V$ we have
\begin{equation}\label{eq:dis}
\nu_{k,p}(G,\mu)^{\frac{1}{p}} \leq \frac{2}{ \mathcal{D}(\{A_{\alpha}\})      }\max_{\alpha=0,\dots,k}  \log \frac{e(\mu(V)-\sum_{\beta\neq \alpha}\mu(A_{\alpha})    )}{\mu(A_{\alpha})}.
\end{equation}
\end{thm}
\begin{rem}
 The case where $k=1$ and $p=2$ were studied in (\cite{AM}). In this
 case Alon-Milman obtained a similar bound for $\lambda_1(G,\mu)$. 
 \end{rem}
\begin{rem}Let us compare Theorem \ref{thm:discrete modified main
 result} with Chung-Grigor'yan-Yau's result (\cite{CGY1,CGY2}, see
 Theorem 3.22 of \cite{Gr}). Let
 $N:=\#V$. Given $A_0,A_1,\dots,A_k\subset V$ such that
 $\mathcal{D}(\{A_{\alpha}\})>1$ we set
\begin{equation*}
 \delta(\{A_{\alpha}\}):=\max_{\alpha\neq
  \beta}\Big(\frac{\mu(V\setminus A_{\alpha})\mu(V\setminus A_{\beta})}{\mu(A_{\alpha})\mu(A_{\beta})}\Big)^{\frac{1}{2(\mathcal{D}(\{A_{\alpha}\})-1)}}.
 \end{equation*}
 We may assume that $\mu(A_{0})\leq \mu(A_{1})\leq \cdots\leq \mu(A_{k})$,
 and thus $\mu(A_{\alpha})\leq \frac{1}{2}\mu(V)$ for $\alpha=0,\dots,k-1$;
 in particular,
 $\delta(\{A_{\alpha}\})\geq 1$. In the case of $p=2$ Chung-Grigor'yan-Yau showed that
 \begin{align}\label{ineq:CGY}
  \lambda_k (G,\mu)\leq \frac{\delta(\{A_{\alpha}\})-1}{\delta(\{A_{\alpha}\})+1}\lambda_{N-1}(G,\mu).
  \end{align}Note that $\frac{N}{N-1}\leq \lambda_{N-1}(G,\mu)\leq 2$
 (Theorem 2.11 of \cite{Gr})
 and hence the
 right-hand side of the above inequality (\ref{ineq:CGY}) is at most $2$ whereas our
 bound in (\ref{eq:dis}) can be greater than $2$. For example if we
 choose $\{A_{\alpha}\}_{\alpha=0}^k$ such that $\mu(A_{\alpha})$ is
 small and $\mathcal{D}(\{A_{\alpha}\})$ is not large then the right-hand
 side of (\ref{eq:dis}) is greater than $2$, hence in this case the
 inequality (\ref{eq:dis}) is not informative. It seems in many choices of
 $\{A_{\alpha}\}$ their inequality (\ref{ineq:CGY}) is better than our
 inequality (\ref{eq:dis}). Our inequality can be better when
 $(\mu(V)-\sum_{\beta \neq \alpha}\mu(A_{\beta}))/\mu(A_{\alpha})\sim 1$ and
 $\mathcal{D}(\{A_{\alpha}\})\gg 1$. For example, let us consider a
 family $\{(V_{\alpha},E_{\alpha})\}_{\alpha=0}^k$ of complete graphs such that
 $\# V_{\alpha}=k$. For each $\alpha=0,1,\dots,k-1$ we choose and fix a vertex
 $x_{\alpha}\in V_{\alpha}$ and connect $x_{\alpha}$ and $x_{\alpha+1}$
 by a path graph $(V_{\alpha}',E_{\alpha}')$ of $\#V_{\alpha}'\sim \log k$. We
 then set
 \begin{align*}
  V:=V_k \cup \bigcup_{\alpha=0}^{k-1}(V_{\alpha}\cup V_{\alpha}')
  \text{ and } E:=E_{k}\cup \bigcup_{\alpha=0}^{k-1} (E_{\alpha}\cup E_{\alpha}').
  \end{align*}Then $G:=(V,E)$ is a finite connected graph and we equip
 $G$ the simple weight $\mu$, i.e., $\mu_{x,y}=1$ iff $x,y$ form an edge
 in $E$. We then normalize $\mu$ so that $\mu(V)=1$. For each $\alpha=0,1,\dots,k$ we set
 $A_{\alpha}:=V_{\alpha}$. Then since $\mu(A_{\alpha})\doteqdot \frac{1}{k+1}$ we have $(1-\sum_{\beta\neq
 \alpha}\mu(A_{\beta}))/\mu(A_{\alpha})\sim 1$ and
 $\mathcal{D}(\{A_{\alpha}\})\sim \log k$. Thus in this graph $G$ our
 inequality (\ref{eq:dis}) implies
 \begin{align}\label{eq:dis2}
  \lambda_k(G,\mu)^{\frac{1}{2}}\lesssim 2/\log k.
  \end{align}On the other hand, since
 \begin{align*}
  \Big(
  \frac{(1-\mu(A_{\alpha}))(1-\mu(A_{\beta}))}{\mu(A_{\alpha})\mu(A_{\beta})}\Big)^{\frac{1}{2(\mathcal{D}(\{A_{\alpha}\})-1)}}
  \sim k^{\frac{1}{\log k}}=e,
  \end{align*}the inequality (\ref{ineq:CGY}) implies
 \begin{align}\label{eq:CGY2}
  \lambda_k(G,\mu)\lesssim \lambda_{N-1}(G,\mu)\sim 1.
  \end{align}Comparing (\ref{eq:dis2}) with (\ref{eq:CGY2}) our bound is
 better in this case.
 \end{rem}

%%%%%%%%%%%%%%%%%%%%%%%%%%%%
%%%%%%%%%%%%%%%%%%%%%%%%%%%%
%%%%%%%%%%%%%%%%%%%%%%%%%%%%

\bigbreak
\noindent
{\it Acknowledgments.}
The authors are grateful to the anonymous referee for valuable comments.
The first author was supported in part by JSPS KAKENHI (17K14179).
The second author was supported in part by JPSJ Grant-in-Aid for Scientific Research on Innovative Areas ``Discrete Geometric Analysis for Materials Design" (17H06460).

%%%%%%%%%%%%%%%%%%%%%%%%%%%%

\end{document}